\documentclass[12pt]{amsart} % for a short document
\makeatletter
\providecommand{\@LN}[2]{}
\makeatother
\calclayout

\usepackage{microtype}
\usepackage[english]{babel}
\usepackage[maxbibnames=5,minbibnames=4,backend=bibtex]{biblatex}
\usepackage{csquotes}
\usepackage{amssymb}
\usepackage{amsmath}
\usepackage{amsthm}
\newtheorem{thm}{Theorem}[section]
\newtheorem{lem}[thm]{Lemma}
\newtheorem{prop}[thm]{Proposition}
\newtheorem{coro}[thm]{Corollary}

\theoremstyle{remark}
\newtheorem{rema}[thm]{Remark}
\newtheorem{exa}[thm]{Example}
\newtheorem{defi}[thm]{Definition}

\DeclareMathOperator\Ad{Ad}

\usepackage{MnSymbol}
\usepackage{tikz}
\usepackage[all,cmtip]{xy}
\usepackage{todonotes}
\usepackage{stmaryrd}    
\title{Integrability of quotients in Poisson and Dirac geometry}
\author{Daniel \'Alvarez}\address{Instituto de Matem\'atica Pura e Aplicada, Estrada Dona Castorina, 110, Jardim Bot\^anico, CEP \texttt{22460320}, Rio de Janeiro, Brasil}
\email{uerbum@impa.br}
\date{} % Delete this line to display the current date
\begin{document}

%\layout

\begin{abstract} We study the integrability of Poisson and Dirac structures that arise from quotient constructions. From our results we deduce several classical results as well as new applications. We also give explicit constructions of Lie groupoids integrating two interesting families of geometric structures: (i) a special class of Poisson homogeneous spaces of symplectic groupoids integrating Poisson groups and (ii) Dirac homogeneous spaces.  
\end{abstract}\maketitle
\tableofcontents % the asterisk means that the contents itself isn't put into the ToC

\section{Introduction} 

Symplectic groupoids have become a key tool in the study of Poisson structures, being of central importance for the problem of quantization \cite{catfel,groqua} as well as other applications \cite{cradifcoh,conlin,craruimar}. An important issue is that, unlike finite-dimensional Lie algebras, which always admit integrations to Lie groups, not every Poisson manifold is ``integrable'' to a symplectic groupoid \cite{weisygr}. The precise conditions for the integrability of Poisson manifolds (and Lie algebroids in general) were found in \cite{crarui,craruipoi}. Although these results solve the abstract problem of integrability, it is still of fundamental importance to identify concrete classes of Poisson manifolds that are integrable and have explicit constructions for their integrating groupoids. Our work has been largely driven by this problem with special focus on Poisson structures (or more general Lie algebroids) that arise from quotient constructions. In this paper we extend some of the integrability results found in \cite{ferigl,intpoihom} and we put them in a common framework.

The study of Poisson structures resulting from reduction by  symmetries is intimately tied with more general geometric objects, known as \emph{Dirac structures} \cite{coubey,coudir}. In simple terms, Dirac structures are to Poisson structures what closed 2-forms are to symplectic forms, and just as symplectic forms naturally arise as quotients of closed 2-forms, Poisson structures are often realized as quotients of Dirac structures. Extending Poisson to Dirac structures is also essential to make sense of their pullbacks \cite{burint}. It often happens that in order to understand the integrability of a Poisson structure obtained by reduction of a Dirac structure \cite{burint} it is easier to check whether the original Dirac structure is integrable and whether one of its integrations gives rise to an integration of the associated Poisson quotient, as e.g. in \cite{intpoihom}. The main goal of this work is to understand how and when the following diagram closes:
\[ \xymatrix{ \{\text{Presymplectic groupoids}\}\ar[d]_{\text{Lie functor}} \ar@{.>}[rrr]^-{\text{Reduction?} }&&&\{\text{Symplectic groupoids}\}\ar[d]^{\text{Lie functor}} \\
\{\text{Reducible Dirac structures}\} \ar[rrr]_-{\text{Reduction}}&&&\{\text{Poisson structures}\}.} \]
It is not hard to find examples of integrable Dirac structures whose quotient Poisson manifold is nonintegrable, see Example \ref{exahopfib}.
The central result in this paper characterizes the integrability of Poisson structures obtained as quotients of Dirac structures (see Theorem \ref{pullquo} below): 

\begin{thm}\label{main} Let $q:S\rightarrow M$ be a surjective submersion and let $\pi$ be a Poisson structure on $M$. Then the following are equivalent:
\begin{enumerate} \item $(M,\pi)$ is integrable;
		\item the pullback Dirac structure $L:=q^!(T^*M)$ is integrable and the inclusion $\ker Tq\subset L$ is integrable by a Lie groupoid morphism $S \times_M S \rightarrow {G}$, where $S \times_M S \rightrightarrows S$ is the submersion groupoid associated to $q$ and $G$ is a presymplectic groupoid. 
%		\item $L$ is integrable and the image of the kernel of the projection $\mathcal{G} (\ker Tq)\rightarrow \text{hol}(\ker Tq)$ under the canonical morphism $\mathcal{G}(\ker Tq) \rightarrow \mathcal{G} (L) $ integrating $\ker Tq \hookrightarrow L$ is an embedded submanifold.  \qed
\end{enumerate} Moreover, if $\Phi:S \times_M S \rightarrow G$ is a Lie groupoid morphism as above and the fibers of $q$ are connected, then $(M,\pi)$ is integrable by the quotient $G/\sim$, where $\sim$ is the equivalence relation on $G$ defined by 
\[ g \sim \Phi(x)g\Phi(y) \]
for all compatible $g\in G$ and $x,y\in S \times_M S$.
\end{thm}
In Section \ref{sec:promai} we prove Theorem \ref{main} along with its generalization to Lie algebroids and their pullbacks, see Theorem \ref{liequo4}. 

Theorem \ref{main} turns out to have several interesting consequences and applications. For instance, Theorem \ref{main} leads to an alternative proof of classical results concerning the integrability of quotients of Poisson manifolds and the Morita invariance of integrability, see Proposition \ref{coired} and Proposition \ref{morinvint}. As new applications, we obtain the integration of certain types of homogeneous spaces in Poisson and Dirac geometry (Section \ref{sec:homspa}).

It is proven in \cite{luphd} that Poisson groups are always integrable. More recently, Poisson homogeneous spaces of Poisson Lie groups \cite{dripoi} were also shown to be integrable \cite{intpoihom}. Our tools allow us to extend these results and consider other types of homogeneous spaces; these applications are conducted in Section \ref{sec:homspa}. In general, {\em Poisson homogeneous spaces} are Poisson manifolds endowed with a transitive Poisson action of a Poisson groupoid. We identify a natural class of integrable Poisson homogeneous spaces of {\em symplectic groupoids} which are in duality with respect to the classical Poisson homogeneous spaces of Poisson groups, see Theorem \ref{poihomspasymgro}.

In another direction, Poisson groups have been generalized to Dirac geometry. A Dirac-Lie group, as in \cite{liedir}, is a group-like object in the category of Manin pairs. Dirac-Lie groups subsume the study of Poisson groups and of the important Cartan-Dirac structures, which are central in the theory of q-Poisson manifolds \cite{purspi}. In this context we obtain Theorem \ref{dirhomint}, which generalizes the results of \cite{intpoihom} to {\em Dirac homogeneous spaces} \cite{robdir,dirhomspa}.

\section{Preliminaries} \subsection{Lie groupoids and Lie algebroids}
We follow mainly the conventions of \cite{moeint,dufzun}. A {\em smooth groupoid} is a groupoid object in the category of not necessarily Hausdorff smooth manifolds such that its source map is a submersion. The structure maps of a groupoid are its source, target, multiplication, unit map and inversion, denoted respectively $\mathtt{s},\mathtt{t},\mathtt{m},\mathtt{u} , \mathtt{i}$. For the sake of simplifying the notation we also denote $\mathtt{m}(a,b)$ by $ab$. By a {\em Lie groupoid} we mean a smooth groupoid such that its base and source-fibers are Hausdorff manifolds. In order to avoid ambiguity when dealing with several groupoids we use a subindex $\mathtt{s}=\mathtt{s}_G$, $\mathtt{t}_G$, $\mathtt{m}_G$ to specify the groupoid under consideration. 

A {\em Lie algebroid} $A$ over a manifold $M$ is a vector bundle $A\rightarrow M$ and a bundle map $\mathtt{a}:A\rightarrow TM$ called the {\em anchor} that satisfy the following properties: $\Gamma(A)$ has a Lie algebra structure $[\,,\,]$ and the Leibniz rule holds:
\[ [u,fv]=f[u,v]+\left(\mathcal{L}_{\mathtt{a}(u)}f\right)v, \] 
for all $u,v\in \Gamma(A)$ and $f\in C^\infty(M)$. See \cite{higmacalg,vai} for the definition of Lie algebroid morphism.

Let $G\rightrightarrows B$ be a Lie groupoid. Its {\em tangent Lie algebroid} $A=A_{G}$ is the vector bundle $A=\ker T\mathtt{s}|_{B}$ with the anchor given by the restriction of $T\mathtt{t}$ and the bracket defined in terms of right invariant vector fields. The construction of the tangent Lie algebroid is functorial: a Lie groupoid morphism induces canonically a Lie algebroid morphism between the associated Lie algebroids. The functor thus induced is called the {\em Lie functor} and we denote it by $\text{Lie} $. When a Lie algebroid is isomorphic to the tangent Lie algebroid of a Lie groupoid it is called {\em integrable}. Not every Lie algebroid is integrable and the general obstructions for integrability were found in \cite{crarui}. If $A$ is an integrable Lie algebroid, we denote by $\mathcal{G} (A)$ its source-simply-connected integration (which is unique up to isomorphism). 

An important result relating Lie groupoids and Lie algebroids, that will be frequently used in this paper, is the following. Let $\phi:A \rightarrow B$ be a Lie algebroid morphism and suppose that $A$ and $B$ are integrable. Then, for every Lie groupoid $K$ integrating $B$, there is a unique Lie groupoid morphism $\Phi:\mathcal{G} (A) \rightarrow K$ such that $\text{Lie} (\Phi)=\phi$. This result is known as {\em Lie's second theorem} \cite{intinfact,moeint}.

A case of interest in this paper is the following: let $q:S \rightarrow M$ be a surjective submersion. Then the fiber product $S \times_M S$ is a Lie groupoid over $S$, where the source and target maps are the projections to $S$ and the multiplication is defined by $\mathtt{m}((x,y),(y,z))=(x,z)$. The Lie groupoid thus obtained is called the {\em submersion groupoid} associated to $q$, its Lie algebroid is isomorphic to the distribution $\ker Tq \hookrightarrow TS$. 

If $\mathfrak{g} $ is a Lie algebra acting on a manifold $M$, we denote by $u_M$ the vector field on $M$ induced by the action of $u\in \mathfrak{g} $ and by $\mathfrak{g}_M\subset TM$ the distribution generated by all such vector fields. An infinitesimal action induces a Lie algebroid structure on the trivial bundle $\mathfrak{g} \times M$ \cite{moeint} called the {\em action Lie algebroid structure} \cite{macgen}. %More generally, a Lie algebroid $A$ over $M$ acts on a map $J:S \rightarrow M$ if there is a Lie algebra morphism $\rho:\Gamma (A) \rightarrow \mathfrak{X}(S)$ such that $X_S:=\rho(X)$ is $q$-related to $\mathtt{a}(X)$ for all $X\in \Gamma (A)$. In this situation, the pullback vector bundle $J^*A$ inherits a Lie algebroid structure over $S$.   

Let $A$ be a Lie algebroid over $M$ and let $q:S\rightarrow M$ be a map such that $Tq$ is transverse to the anchor of $A$. The {\em pullback Lie algebroid} of $A$ along $q$, denoted $q^! A$ \cite{higmacalg}, is the vector bundle $q^!A:=\{X\oplus U\in TN \oplus q^*A:Tq(X)=\mathtt{a}(U)\}$ endowed with the projection to $TS$ as the anchor and with the Lie bracket given by $[X\oplus fU,Y\oplus gV]=[X,Y]\oplus \mathcal{L}_Xg V-\mathcal{L}_YfU+fg[U,V] $, where $U,V$ are pullbacks of sections of $A$ and $f,g\in C^\infty(S)$. If $q$ is a surjective submersion, we have that the natural inclusion $\ker Tq \hookrightarrow q^!A$ is a Lie algebroid morphism.

\subsection{Poisson structures and symplectic groupoids}\label{subsec:poi} A {\em Poisson structure} on a manifold $M$ is a bivector field $\pi\in \Gamma(\wedge^2TM)$ such that $[\pi,\pi]=0$, where $[\,,\,]$ is the Schouten bracket. A Poisson structure $\pi$ turns $(M,\pi)$ into a \emph{Poisson manifold}. Equivalently, a Poisson structure on $M$ is determined by a Lie algebra structure on $C^\infty(M)$ such that the Leibniz rule holds: $\{fg,h\}=f\{g,h\}+\{f,h\}g$ for all $f,g,h\in C^\infty(M)$. In terms of a bivector field, the bracket $\{\,,\,\}$ is described as $\{f,g\}=\pi(df,dg)$ for all functions $f,g$ on $M$. The cotangent bundle of a Poisson manifold is endowed with a Lie algebroid structure in the following way \cite{dufzun}. Let $\pi^\sharp:T^*M \rightarrow TM$ be the map defined by $\alpha \mapsto i_\alpha \pi$. The Lie bracket on $\Omega^1(M)$ is given by:
\begin{align*}  [\alpha,\beta]=\mathcal{L}_{\pi^\sharp(\alpha)}\beta-\mathcal{L}_{\pi^\sharp(\beta)}\alpha-d\pi(\alpha,\beta),  \end{align*}
for all $\alpha,\beta \in \Omega^1(M)$; the anchor of $T^*M$ is the map $\pi^\sharp$. In this situation, we shall refer to $T^*M$ as the {\em cotangent Lie algebroid} of $M$. A Poisson morphism $J:(P,\pi_P) \rightarrow (Q,\pi_Q)$ between Poisson manifolds is a smooth map that satisfies $J^*\{f,g\}=\{J^*f,J^*g\}$ for all $f,g\in C^\infty (Q)$.

A Lie groupoid $G \rightrightarrows B$ is a \emph{symplectic groupoid} \cite{weisygr,karpoi} if it is endowed with a symplectic form $\omega $ such that the graph of the multiplication in $G\times G\times {G}$ is Lagrangian with respect to the symplectic form $\text{pr}_1^* \omega +\text{pr}_2^* \omega -\text{pr}_3^* \omega $. The base of a symplectic groupoid inherits a unique Poisson structure such that the target map is a Poisson morphism and the cotangent Lie algebroid of this Poisson structure is canonically isomorphic to the Lie algebroid of the symplectic groupoid \cite{macxu2}. If the cotangent Lie algebroid of a Poisson manifold $M$ is integrable, we shall say that $M$ is {\em integrable}. If $M$ is an integrable Poisson manifold, we denote by $\Sigma(M) \rightrightarrows M$ its source-simply-connected integration which automatically inherits a symplectic groupoid structure \cite{macxu2}.  
\subsection{Dirac structures and presymplectic groupoids}\label{subsec:dirpre} Let $M$ be a manifold. We denote by $\mathbb{T}M$ the direct sum $TM \oplus T^*M$. The {\em Courant-Dorfman bracket} \cite{coubey,coudir} on $\Gamma(\mathbb{T}M )$ is:
\[ \llbracket X+\alpha,Y+\beta \rrbracket:=[X,Y]+\mathcal{L}_X\beta-i_Yd\alpha, \]
for all $X,Y\in \Gamma(TM)$, $\alpha,\beta\in \Gamma(T^*M)$. Define the bilinear pairing   
			\[ \langle X+\alpha,Y+\beta\rangle=\langle\alpha,Y\rangle+\langle \beta,X \rangle; \]
a subbundle $F\subset \mathbb{T} M $ is called {\em isotropic} if $\langle\,,\, \rangle|_F=0$. A maximally isotropic subbundle of $\mathbb{T} M $ is called {\em Lagrangian}. A {\em Dirac structure} is a Lagrangian subbundle $L\subset \mathbb{T} M $ which is involutive with respect to the restricted Courant-Dorfman bracket. If $L \subset \mathbb{T}M $ is a Dirac structure, $(M,L)$ is called a {\em Dirac manifold}.

Take a Dirac structure $L$ on $M$. A function $f$ on $M$ is {\em admissible} if there exists a vector field $X_f$ such that $X_f+df\in \Gamma(L)$. Admissible functions form a Poisson algebra with the bracket $\{f,g\}=\mathcal{L}_{X_f}g$. If the distribution $D=L \cap(TM\oplus 0)$ is of constant rank, its induced foliation is called {\em the null foliation of $L$}. The admissible functions are the functions constant along the leaves of $D$, so they are identified with the smooth functions on the leaf space of $D$ when this is smooth. A Dirac structure is called {\em reducible} if its null foliation is simple.  

Take a smooth map $f:(M,L_M) \rightarrow (N,L_N)$ between Dirac manifolds. We say that $f$ is a \emph{forward Dirac map} \cite{burint,purspi} if
			\begin{align*} L_N=\mathfrak{F}_f(L_M):=\{Tf(X)\oplus  \alpha :X\oplus f^*\alpha \in L_M\}; \end{align*}
				We say that $f$ is a \emph{backward Dirac map} if				\begin{align*} L_M=\mathfrak{B}_f(L_N):=\{X\oplus  f^*\alpha :Tf(X)\oplus \alpha \in L_N\}. \end{align*}
 
Generalizing the correspondence between Poisson manifolds and symplectic groupoids, presymplectic groupoids integrate Dirac structures \cite{burint}. A 2-form $\omega $ on a Lie groupoid $G \rightrightarrows M$ is multiplicative if $ \text{pr}_1^* \omega +\text{pr}_2^* \omega -\text{pr}_3^* \omega $ vanishes on the graph of the multiplication map in $G \times G\times G$. A {\em presymplectic groupoid} is a Lie groupoid $G \rightrightarrows M$ endowed with a multiplicative 2-form $\omega $ such that $d \omega =0$, $\dim G=2\dim M$ and $\ker \omega \cap \ker T \mathtt{s} \cap \ker T \mathtt{t} =0$. If $(G, \omega ) \rightrightarrows M$ is a presymplectic groupoid, then there is a unique Dirac structure $L$ on $M$ such that $\mathtt{t}:(G, \text{graph}(\omega )) \rightarrow (M,L)$ is a forward Dirac map and $L$ is canonically isomorphic to the Lie algebroid of $G$ \cite{burint}. 

We can see that a bivector field $\pi$ on $M$ is Poisson if and only if $\text{graph}(\pi)\hookrightarrow \mathbb{T}M$ is a Dirac structure; in such a situation, if $q:S \rightarrow M$ is a submersion then $q^!(T^*M)$ is identified with a Dirac structure on $S$ such that $q$ is a backward Dirac map.

\section{Proof of the main result}\label{sec:promai} \subsection{The case of Lie algebroids} We begin this section by establishing a more general version of Theorem \ref{main} which corresponds to the relationship between the integrability of a Lie algebroid and the integrability of its pullback Lie algebroid along a surjective submersion, see Theorem \ref{liequo4}. Later we specialize the argument for the case of Poisson and Dirac structures thus obtaining a slightly more refined integrability criterion in Theorem \ref{pullquo} below. Theorem \ref{liequo4} answers the following question: {\em given a surjective submersion $q:S \rightarrow M$ and a Lie algebroid $A$ over $M$, how is the integrability of $A$ related to the integrability of $q^!A$?} It is not enough for $q^!A$ to be integrable in order to conclude that so is $A$, see Example \ref{exahopfib}.
 
\begin{defi} Let $q:S\rightarrow M$ be a surjective submersion and let $A$ be a Lie algebroid on $M$. We say that $q^!A$ admits a {\em $q$-admissible integration} if $q^! A$ is integrable and the inclusion $\ker Tq \hookrightarrow  q^!A$ is integrable by a Lie groupoid morphism $S \times_M S \rightarrow {G}$, where $S \times_M S \rightrightarrows S $ is the submersion groupoid. \end{defi}
The condition above generalizes the situation of principal bundles considered in \cite[Definition 2.6]{intpoihom}. Let us stress that the distribution $\ker Tq$ is always integrable by its monodromy groupoid $\mathcal{G} (\ker Tq)$ and Lie's second theorem \cite{moeint} implies that, if $q^!A$ is integrable, then the inclusion $\ker Tq \hookrightarrow  q^!A$ is integrable by a Lie groupoid morphism $\mathcal{G} (\ker Tq) \rightarrow \mathcal{G} (q^!A)$. In general, such a Lie groupoid morphism does not induce a morphism with source the submersion groupoid $S \times_M S \rightrightarrows S$, see Proposition \ref{exapullquo}. Let us define the following equivalence relation $R$ on $G$: 
\begin{align}  g \sim \Phi(x)g\Phi(y) \label{equrel} \end{align} 
for all compatible $g\in G$ and $x,y\in S \times_M S$.
\begin{thm}\label{liequo4} Let $q:S\rightarrow M$ be a surjective submersion and let $A$ be a Lie algebroid on $M$. Then the following are equivalent:
\begin{enumerate} \item $A$ is integrable.
		\item the pullback Lie algebroid $q^!A $ admits a $q$-admissible integration $G$. 
%		\item $L$ is integrable and the image of the kernel of the projection $\mathcal{G} (\ker Tq)\rightarrow \text{hol}(\ker Tq)$ under the canonical morphism $\mathcal{G}(\ker Tq) \rightarrow \mathcal{G} (L) $ integrating $\ker Tq \hookrightarrow L$ is an embedded submanifold.  \qed
\end{enumerate} Moreover, if $\Phi:S \times_M S \rightarrow G$ is a Lie groupoid morphism such that $\text{Lie} (\Phi)$ is the inclusion $\ker Tq \hookrightarrow q^!A$, then the quotient $G/R$ inherits a unique Lie groupoid structure over $M$ such that the quotient map $G \rightarrow G/R$ is a Lie groupoid morphism. In this situation, $\text{Lie} (G/R)\cong A$. \end{thm} 
\begin{rema} The previous result has been observed in the particular case in which $q:S \rightarrow {M}$ is a principal bundle and hence $S \times_M S \rightrightarrows S$ is an action groupoid \cite{ferigl,braferdir,blowei,ferstrcar,intpoihom}. In this situation, the quotient map $G \rightarrow G/R$ is a {\em Morita fibration} \cite{riemetsta}. \end{rema}
 Let us explain the main steps in the proof of Theorem \ref{liequo4}. 

{\em Step 1.} The fact that condition 1 implies condition 2 of the previous theorem is a consequence of the construction of {\em pullback Lie groupoids}.
\begin{defi}[\cite{higmacalg}] Let $q: S \rightarrow M$ be a submersion and let $K \rightrightarrows M$ be a Lie groupoid, then the {\em pullback Lie groupoid} $q^!K \rightrightarrows S$ is the Lie groupoid 
		\[ q^!K:=\{(x,g,y)\in (S \times_M S) \times {K} \times( S \times_M S):q(x)=\mathtt{t}(g),q(y)=\mathtt{s}(g)\}, \] 
where the source and target are the projections to $S$ and the multiplication is given by $\mathtt{m}( (x,a,y),(y,b,z))=(x,\mathtt{m}(a,b),z)$. \end{defi} 
If ${A}$ is integrable by some Lie groupoid $K$, then $q^!A$ is integrable by the pullback groupoid $q^!K$, see \cite[Corollary 1.9]{higmacalg}. We have that $q^!K$ is a $q$-admissible integration of $q^!A$: we simply define the map $\Phi:S \times_M S \rightarrow q^!K$ by means of $\Phi(x,y)=(x,\mathtt{u}(q(x)),y)$. 

{\em Step 2.} If condition 2 holds, we shall see that the equivalence relation $R$ is induced by a {\em smooth congruence} on $G$ \cite{macgen}. Let us denote $S \times_M S \rightrightarrows S$ by $H \rightrightarrows S$ and let us consider the fiber product $(H \times H) \times_{M \times M} G:=\{(x,g,y)\in H \times G \times H:\mathtt{s}_H(x)=\mathtt{t}_G(g),\mathtt{t}_H(y)=\mathtt{s}_G(g)    \}$. The map
	\begin{align} T:(H \times H) \times_{M \times M} G \rightarrow G \times G, \quad T(x,g,y)=(\Phi(x)g\Phi(y),g) \label{comdougro} \end{align}   
	is an injective proper immersion whose image is the equivalence relation $R\subset G \times G$. In particular, $R $ is a groupoid over $G$. Notice that $R \rightrightarrows H$ is also a subgroupoid of the product Lie groupoid $G \times G \rightrightarrows S \times S$. This fact guarantees that there is a unique (algebraic) groupoid structure on $G/R$ such that the quotient map $Q:G \rightarrow G/R$ is a groupoid morphism. If $G$ is Hausdorff, then $T$ is a closed embedding and so Godement's Theorem \cite{serlie} implies that $ G/R$ is a Hausdorff smooth manifold. Therefore, we get that $G/R \rightrightarrows M$ admits a unique Lie groupoid structure such that $Q$ is a Lie groupoid morphism, see \cite[Thm. 2.4.6]{macgen} (notice that a Lie groupoid in the sense of \cite{macgen} is assumed to be Hausdorff). If $G$ is not Hausdorff, it is not automatic that $R$ is an embedded submanifold. For the general case we shall use instead the following argument: we shall produce a slice for the orbits of $R$ on $G$ at every point in $G$ and we shall see that the collection of these slices defines a smooth structure on  $G/R$. This argument shall be described in the following couple of lemmas.

First of all, we need a simple property of pullback Lie algebroids.
\begin{lem}\label{liequo} Let $\phi:B \rightarrow A$ be fibre-wise surjective Lie algebroid morphism over a submersion $q:S\rightarrow M$. Suppose that $\mathtt{a}_B $ defines an isomorphism $\ker \phi\cong\ker Tq$. Then $B\cong q^!A$. \end{lem} 
	\begin{proof} The map $\chi:B \rightarrow q^!A$ defined by $\chi(u)=(\mathtt{a}(u),\phi(u))$ is a morphism of Lie algebroids thanks to the universal property of $q^!A$ \cite[Proposition 1.8]{higmacalg}. We have that $\chi$ is injective since $\mathtt{a}_B|_{\ker \phi}$ is injective. Since $B$ and $q^!A$ have the same rank, $\chi$ is an isomorphism. \end{proof}

\begin{lem}\label{liequo3} Suppose that $\Phi:S \times_M S \rightarrow G$ is a Lie groupoid morphism such that $\text{Lie} (\Phi)$ is the inclusion $\ker Tq \hookrightarrow q^!A$. Then the quotient $G/R$ inherits a unique Lie groupoid structure such that the quotient map $G \rightarrow G/R$ is a Lie groupoid morphism and $\text{Lie} (G/R)\cong A$. \end{lem} 
	\begin{proof} Let us denote $G/R$ by $\overline{G}$. The quotient map $Q:G \rightarrow \overline{G} $ induces an isomorphism of topological groupoids $G\cong q^! \overline{G} $ given by the map $g\mapsto (\mathtt{t}_G(g),Q(g),\mathtt{s}_G(g)) $ (which is an open continuous bijection and hence a homeomorphism). Since $q$ is a submersion, we can find a chart for every point in $S$ such that $q$ looks like a projection; let $V_i=U_i \times F$ be such charts for $i=1,2$ and suppose that $q|_{V_i}$ is identified with the projection to $U_i$. Then we have that $W:=\mathtt{s}_G^{-1}(V_1)\cap \mathtt{t}_G^{-1}(V_2) $ is homeomorphic to 
\[ \{ (\overline{\mathtt{t}}(k),x,k,\overline{\mathtt{s} }(k),y )\in V_1 \times \overline{G} \times V_2|x,y\in F, k\in \overline{\mathtt{s}}^{-1}(U_1)\cap \overline{\mathtt{t}}^{-1}(U_2)\}\subset q^! \overline{G}, \]  
where $\overline{\mathtt{s}  }, \overline{\mathtt{t} } $ denote the source and target map induced on $\overline{G} $. This means that $\overline{\mathtt{s}}^{-1}(U_1)\cap \overline{\mathtt{t}}^{-1}(U_2)$ is defined as the inverse image of a point under the submersion 
\[ (p_1,p_2)\circ (\mathtt{t}, \mathtt{s}) :  W \rightarrow F \times F, \]
where $p_i:V_i \rightarrow F$ is the projection. Then the open subset $\overline{\mathtt{s}}^{-1}(U_1)\cap \overline{\mathtt{t}}^{-1}(U_2)$ of $\overline{G}$ inherits a smooth structure; we shall see that this smooth structure does not depend on the identification that we used. First of all, a change of submersion charts for $q$ would change the smooth structure induced on $\overline{\mathtt{s}}^{-1}(U_1)\cap \overline{\mathtt{t}}^{-1}(U_2)$ by a diffeomorphism so it is independent of the choice of submersion charts. Let us see that the smooth structure on $\overline{\mathtt{s}}^{-1}(U_1)\cap \overline{\mathtt{t}}^{-1}(U_2)$ does not depend on the points of the $q$-fibers that we picked in order identify it with a submanifold of $G$. Suppose that on an open set $V'_1\subset S$ such that $q(V_1)=q(V_1')\cong U_1$ there is another submersion chart which allows us to identify it with $U_1 \times F'$. For every $(a,b)\in F \times F'$ define a map 
\[ \chi:\mathtt{s}^{-1}(\text{pr}_2^{-1}(a))\cap \mathtt{t}^{-1}(V_2) \rightarrow \mathtt{s}^{-1}(\text{pr}_2^{-1}(b))\cap \mathtt{t}^{-1}(V_2)\] 
by $\chi(g)=\mathtt{m}( g,\phi(\mathtt{s}(g),y))$, where $(\mathtt{s}(g),y)\in V_1 \times V_1'$ is defined by the conditions $q(\mathtt{s}(g))=q(y)$ and $\text{pr}_2(y)=b$ for $\text{pr}_2:V_1' \rightarrow F'$ the projection. We have that $\chi$ is a diffeomorphism which restricts to a diffeomorphism between the smooth structures induced on $\overline{\mathtt{s}}^{-1}(U_1)\cap \overline{\mathtt{t}}^{-1}(U_2)$ as above. This implies that they do not depend on the choice of a point on the $q$-fiber. We proceed analogously with a different choice of $V_2$. Since $W$ as above is diffeomorphic to $F \times \overline{\mathtt{s}}^{-1}(U_1)\cap \overline{\mathtt{t}}^{-1}(U_2) \times F$, the restricted quotient map $Q|_W:W \rightarrow  \overline{\mathtt{s}}^{-1}(U_1)\cap \overline{\mathtt{t}}^{-1}(U_2)$ is a submersion. As a consequence, $\overline{G}$ inherits a (unique) smooth structure such that $Q$ has the universal property of a quotient, see \cite[Chapter III]{serlie}. 

Let us check that $\overline{\mathtt{s} }$-fibers in $\overline{G} $ are Hausdorff. Take $C=\mathtt{s}_G^{-1}(x) \times {\mathtt{s} }_G^{-1}(x)\subset G \times G$; then $f:=T|:T^{-1}(C)\cong H \times_S(\mathtt{s}^{-1}_G(x)) \rightarrow C$ is also an injective proper immersion since $C$ is closed, where $T$ is the map \eqref{comdougro}. Since $\mathtt{s}^{-1}_G(x) $ is Hausdorff, $f$ is a closed embedding and so Godement's Theorem implies that the quotient space $\overline{\mathtt{s}}^{-1}(q(x)) $ is a Hausdorff manifold as well. The source, target, unit and inversion maps descend to smooth maps between $\overline{G}$ and $M=S/H$ thanks to the universal property of the quotient. In particular, we get that the induced source map $\overline{\mathtt{s} }:\overline{G} \rightarrow M$ makes the following diagram commute:
\[ \xymatrix{G \ar[r]^-Q\ar[d]_{\mathtt{s}_G}  &\overline{G} \ar[d]^{\overline{\mathtt{s}  }}  \\ S \ar[r]_q &M. } \]  
By taking derivatives we get that $Tq\circ T \mathtt{s}_G =T \overline{\mathtt{s} }\circ T Q$. Since the left hand side is fiber-wise surjective, so is $T \overline{\mathtt{s} } $ and hence  $ \overline{\mathtt{s}}$ is a submersion. 
 
By counting dimensions we get that $(Q,Q)|_{G \times_S G}$ is also a submersion and hence $\mathtt{m}_G$ induces a smooth multiplication map on $ \overline{G}$:
\[ \xymatrix{G \times_S G \ar[d]_{(Q,Q)}\ar[r]^-{\mathtt{m}_G}  &G\ar[d]^-Q    \\ \overline{G} \times_M\overline{G} \ar[r]_-{\overline{\mathtt{m}}}& \overline{G} .} \]
Therefore, $\overline{G} $ is a Lie groupoid. 

Finally, $Q:G \rightarrow \overline{G}$ induces a morphism of Lie algebroids $\text{Lie}(G)\rightarrow\text{Lie}(\overline{G})$ to which it is associated a surjective vector bundle map $\text{Lie}({G} )\rightarrow q^*\text{Lie}(\overline{G})$ with kernel $\text{Lie}({H} )$. Then the isomorphism $q^!\text{Lie}(\overline{G})\cong \text{Lie}(G)  $ follows from Lemma \ref{liequo}. \end{proof} 
\begin{rema}\label{rem:embsub} In the context of Lemma \ref{liequo4}, we have that $G/R$ is smooth so it follows from Godement's Theorem that $R$ is indeed an embedded submanifold of $G \times G$ and then $R \rightrightarrows G$ is a Lie subgroupoid of the pair groupoid $G \times G \rightrightarrows G$. Let us notice that $R \rightrightarrows G$ can also be seen as an action groupoid $(H \times H) \ltimes G \rightrightarrows G$ with the product groupoid $H \times H \rightrightarrows S \times S$ acting on $(\mathtt{t}_G,\mathtt{s}_G):G \rightarrow S \times S$ via \eqref{equrel}. \end{rema} 			
	\begin{proof}[Proof of Theorem \ref{liequo4}] It only remains to check that condition 2 implies condition 1, so suppose that condition 2 holds. Lemma \ref{liequo3} implies that $G/R$ is a lie groupoid such that $q^!A\cong q^!\text{Lie} (G/R)$ but by definition of the bracket in $q^!A$ it follows that $A\cong \text{Lie} (G/R)$. \end{proof}

\subsection{The case of reducible Dirac structures} As motivation for Theorem \ref{main} (Theorem \ref{pullquo} below), let us consider the following situation. 
\begin{exa}[\cite{intpoihom,craruipoi}]\label{exahopfib} Let $M$ be $\mathbb{S}^2 \times \mathbb{R}$ with the Poisson structure $\pi$ given by multiplying the canonical area form $\omega_0$ on $\mathbb{S}^2$ by a positive function $f\in C^\infty(\mathbb{R})$ and let $S$ be $\mathbb{S}^3 \times \mathbb{R}  $. If $q:S \rightarrow M$ is $p \times \text{id}_{\mathbb{R}} $, where $p: \mathbb{S}^3 \rightarrow \mathbb{S}^2$ is the Hopf fibration, then we have that $L=q^! (T^*M)$ is always integrable while $M$ does not have to be, see Proposition \ref{exapullquo}. \end{exa} 

In this subsection we shall consider the following refinement of Theorem \ref{liequo4} that takes the existence of (pre)symplectic forms into account. 
\begin{defi} Let $G \rightrightarrows S$ be a $q$-admissible integration of a Dirac structure $L$ on $S$, where $q:S \rightarrow M$ is a surjective submersion, and let $\Phi: S \times_M S \rightarrow G $ be the associated Lie groupoid morphism. If $G$ is endowed with a presymplectic form $\omega $ whose infinitesimal counterpart is provided by $L$ and is such that $\Phi^* \omega =0$, then we say that $G$ is a {\em $q$-admissible presymplectic integration of $L$}. \end{defi} This concept generalizes \cite[Definition 2.9]{intpoihom}. If the $q$-fibers are connected, then the submersion groupoid $S \times_M S \rightrightarrows S$ is source-connected and hence the condition $\Phi^* \omega =0$ is automatic \cite{burcab}.  
\begin{thm}\label{pullquo} Let $q:S\rightarrow M$ be a surjective submersion and let $(M,\pi)$ be a Poisson manifold. Then the following are equivalent.
\begin{enumerate} \item $(M,\pi)$ is integrable.
		\item the pullback $L:=\mathfrak{B}_q(\text{graph}(\pi))   $ admits a $q$-admissible presymplectic integration $(G,\omega )$.   
\end{enumerate} 
Moreover, if $\Phi:S \times_M S \rightarrow G$ is the Lie groupoid morphism associated to a $q$-admissible presymplectic integration, then there is a symplectic form $\overline{\omega }\in \Omega^2(G/R)$ such that $Q^* \overline{\omega }=\omega $, where $Q:G  \rightarrow G/R $ is the quotient map and $R$ is given by \eqref{equrel}. \end{thm} 
\begin{defi} Let $\mathcal{G} \rightrightarrows P$ be a Lie groupoid and let $\omega \in \Omega^2(P)$ be closed. We say that $\omega $ is $\mathcal{G} $-basic if (1) $\ker \omega $ coincides with the tangent distribution to the $\mathcal{G} $-orbits in $P$ and (2) $\mathtt{t}^* \omega =\mathtt{s}^* \omega $. \end{defi}
If $\omega $ is a $\mathcal{G} $-basic 2-form, it follows from considering the representation of $\mathcal{G} $ on the normal bundle of a $\mathcal{G} $-orbit that it induces a symplectic form on the orbit space of $\mathcal{G} $, provided it is a smooth manifold.
\begin{proof}[Proof of Theorem \ref{pullquo}] If $\mathcal{G} \rightrightarrows M$ is a symplectic groupoid, then $q^! \mathcal{G} $ is a $q$-admissible presymplectic integration of $L$: we just take the pullback of the symplectic form on $\mathcal{G} $ along the projection map $q^! \mathcal{G} \rightarrow \mathcal{G} $ which is a Lie groupoid morphism.  

Let us call $H=S \times_M S$. Suppose that $(G,\omega )$ is a $q$-admissible integration of $L$ and let $\Phi:H\rightarrow {G}  $ be the associated Lie groupoid morphism which integrates the inclusion of Lie algebroids $\ker Tq \hookrightarrow L$. In this situation, $R \rightrightarrows {G} $ is a Lie groupoid, see Remark \ref{rem:embsub}, Let us denote the source and target of this Lie groupoid by $\mathtt{s}^G$ and $\mathtt{t}^G$ respectively, in order not to confuse them with the corresponding maps for $R \rightrightarrows H$, seen as a Lie subgroupoid of $G \times G \rightrightarrows S \times S$. We have that the $R$-orbits on $G$ are tangent to the kernel of ${\omega}$. In fact, the vectors of the form $u^r-v^l$ for $u,v\in \ker Tq$ generate the tangent spaces to the $R$-orbits. But these tangent vectors generate the kernel of ${\omega}$, see \cite[Proposition 2.2]{intpoihom}. On the other hand, since $\Phi^* \omega =0$ and $\omega $ is multiplicative, it follows that $(\mathtt{t}^G)^* \omega =(\mathtt{s}^G)^* \omega $. Therefore, $\omega $ is basic with respect to $R \rightrightarrows {G} $ and so the quotient Lie groupoid $G/R$ (Theorem \ref{liequo4}) inherits a symplectic form $\overline{\omega }$ with the desired property. \end{proof}
\subsection{Some immediate applications} We present three straightforward applications of our main result.
\paragraph{(1)} We study the integrability of Example \ref{exahopfib} in the next proposition; the following result has been obtained with other techniques in \cite{craruipoi,catfel}. Compared to those treatments our approach seems more elementary although it uses some properties of the van Est map \cite{weixu,cradifcoh}. The purpose of the following argument is to illustrate in a concrete situation what is the obstruction for the existence of a Lie groupoid morphism satisfying condition 2 of Theorem \ref{liequo4}. 
\begin{prop}[\cite{catfel}]\label{exapullquo} Let $M$ be $\mathbb{S}^2 \times \mathbb{R}$ with the Poisson structure $\pi$ given by multiplying the canonical area form $\omega_0$ on $\mathbb{S}^2$ by a positive function $f\in C^\infty(\mathbb{R})$. Then $\pi$ is integrable if and only if $f$ is constant or if it does not have any critical points. \end{prop} 
Let $S$ be $\mathbb{S}^3 \times \mathbb{R}$ and let $q: S \rightarrow M$ be the Hopf fibration $p: \mathbb{S}^3 \rightarrow \mathbb{S}^2$ times the identity on $\mathbb{R}$. The Dirac structure $L=q^! (\text{graph}(\pi))$ is always integrable and even explicitly so, see \cite[\S 8]{burint}. So the inclusion $\ker Tq \hookrightarrow L$ is integrable by a Lie groupoid morphism $\Psi:\mathcal{G}(\ker Tq) \hookrightarrow \mathcal{G}(L)$. Thanks to \cite[Proposition 4.4]{braferdir}, we know that such a morphism descends to a Lie groupoid morphism as in condition 2 of Theorem \ref{liequo4} precisely when the image under $\Psi$ of the isotropy groups of $\mathcal{G}(\ker Tq)$ constitutes an embedded submanifold of $\mathcal{G}(L)$. So we just have to check that this last property is equivalent to $f$ being constant or not having any critical point.
\begin{proof}[Proof of Proposition \ref{exapullquo}] Consider the monodromy groupoid $\mathcal{G} (\mathcal{F} )$ of the foliation $\mathcal{F} $ induced by the projection of $L$ on $TS$. We have that $\mathcal{G} (\mathcal{F})$ is isomorphic to the pair groupoid $\mathbb{S}^3 \times \mathbb{S}^3 \rightrightarrows \mathbb{S}^3$ times the unit groupoid on $\mathbb{R}$. There is an action of $\mathcal{G} (\mathcal{F})$ on the conormal bundle $\nu^*$ of $\mathcal{F}$ which is trivial in this case and we can see that $\mathcal{G} (L)$ is isomorphic to the twisted action groupoid $\mathcal{G} (\mathcal{F}) \ltimes_c \nu^*$, where $c$ is a differentiable 2-cocycle defined as follows. There is a natural transformation $\Phi$ between differentiable cohomology and Lie agebroid cohomology called the van Est map \cite{weixu} which may be defined at the cochain level by an explicit morphism of chain complexes \cite{cradifcoh}: 
\[ \Phi: (C^\bullet(\mathcal{G},E),\delta) \rightarrow (\Omega^\bullet (\text{Lie} (\mathcal{G} ),E),d), \]
where $\mathcal{G} $ is any Lie groupoid and $E$ a representation of $\mathcal{G} $. In our case, the leafwise 2-form $fq^* \omega_0 \in \Omega^2(\mathcal{F} )$ induces a 2-form $ \Omega^2(\mathcal{F},\nu^*)$ as $\omega :=d_{dR}(f q^* \omega_0)=df\wedge q^* \omega_0$, where $d_{dR}$ is the full de Rham differential on $S$. The 2-cocycle $c\in C^2(\mathcal{G} (\mathcal{F}),\nu^* )$ is induced by the 2-form $\omega\in \Omega^2(\mathcal{F},\nu^*) $ as follows. Let $c_0\in C^2(\mathbb{S}^3 \times \mathbb{S}^3,\mathbb{R})$ be such that $\Phi( c_0)=p^* \omega_0$. Since any pair groupoid is proper, its differentiable cohomology vanishes in positive degree \cite{cradifcoh} and so there is $s\in C^1(\mathbb{S}^3 \times \mathbb{S}^3,\mathbb{R})$ such that $\delta s=c_0$. Hence $\eta:=\Phi(s)\in \Omega^1(\mathbb{S}^3)$ is such that $d\eta =q^* \omega_0$. By the explicit formula of $\Phi$ we can see then that $\Phi(df \cdot c_0 )=df \wedge q^* (\omega_0)$ and so we can take $c=df\cdot c_0$. The twisted multiplication on $\mathcal{G} (\mathcal{F} )\ltimes_c \nu^*$ is defined as \[ \mathtt{m} ((g,u),(h,v))=(gh,u+v+c(g,h)); \]
for every $g,h \in \mathcal{G} (\mathcal{F} )$ for which the multiplication is defined. The morphism $\Psi: \mathbb{R} \times S \rightarrow \mathcal{G} (\mathcal{F} )\ltimes_c \nu^*$ induced by Lie's second theorem applied to the inclusion $\ker Tq \hookrightarrow L$ is defined then by the formula 
\[ (t,x,u)\mapsto \left(e^{it}x,x,u,\frac{df}{du}(u) \left(\int_0^t \eta(X_{\mathbb{S}^3} )|_{e^{iv}x} dv-s(e^{it}x,x)\right)\wedge du \right); \]    
for all $(t,x,u)\in \mathbb{R} \times \mathbb{S}^3 \times \mathbb{R}=\mathbb{R} \times S  $, where $X_{\mathbb{S}^3}$ is the vector field induced by the canonical generator of $\text{Lie} (\mathbb{R})$. Let us note that the integral in the previous formula does not vanish for $t=2\pi$ otherwise $\omega_0$ would be exact as well. Let us denote $\mathcal{M}:=\Psi(2\pi \mathbb{Z} \times {S} )\subset \mathcal{G} (\mathcal{F} )\ltimes_c \mathbb{R}$. If $f$ is constant, then $\mathcal{M}\cong \mathtt{u}( S) \hookrightarrow \mathcal{G} (\mathcal{F} )\ltimes_c \mathbb{R} $. If $f$ is non constant and has no critical points, then $\mathcal{M}\cong \mathbb{S}^3 \times \text{graph}(df/du) \times \mathbb{Z}$ is a submanifold of $\mathcal{G} (\mathcal{F} )\ltimes_c \mathbb{R}\cong \mathbb{S}^3 \times \mathbb{S}^3 \times  \mathbb{R} \times \mathbb{R} $. On the other hand, if $f$ is non constant and has some critical point, then $\mathcal{M} $ fails to be a manifold. Indeed, $\mathcal{M} $ is homeomorphic to $\mathbb{S}^3 \times \mathcal{X}  $, where $\mathcal{X}\cong\{(u, n\frac{df}{du}(u) )\in \mathbb{R}^2 |(u,n)\in \mathbb{R}\times \mathbb{Z}\}  $ and the graphs of the functions $u \mapsto n \frac{df}{du}(u) $ do not coincide but have a common point where $df/du$ vanishes. Since the $\mathbb{S}^1$-action on $S$ is Hamiltonian with respect to $L$ with constant moment map $S \rightarrow \mathbb{R}$ \cite{braferdir}, we have that the result follows from \cite[Proposition 4.4]{braferdir}. \end{proof}
\paragraph{(2)} Coisotropic reduction appears frequently in symplectic and Poisson reduction. The following result is concerned with the integrability of a quotient obtained by this method. A corollary of this result is \cite[Thm. 3.6]{ferigl}.
\begin{prop}\label{coired} Let $C$ be a coisotropic submanifold of a Poisson manifold $P$. Suppose that $\pi|_{T^\circ C}$ is injective and the coisotropic reduction $\overline{C}$ of $C$ is a smooth manifold. Then $\overline{C}$ is integrable if and only if the inclusion of Lie algebroids $T^\circ C\hookrightarrow \pi^\sharp(T^\circ C)^\circ$ is integrable by a Lie groupoid morphism with source $C \times_{\overline{C}} C \rightrightarrows C$. \qed\end{prop} 
\paragraph{(3)} As the last application of this section we discuss the weak Morita invariance of integrability.
		\begin{defi}\cite{gingro} Let $A_i$ be Lie algebroids over $M_i$, $i=1,2$. Then $A_1$ and $A_2$ are {\em weakly Morita equivalent} if there is a manifold $S$ and surjective submersions $q_i:S\rightarrow M_i$ with simply-connected fibers such that the pullback Lie algebroids $q_i^! A_i$ are isomorphic. \end{defi}  
\begin{prop}\label{morinvint} If two Lie algebroids are weakly Morita equivalent, then one is integrable if and only if the other one is; moreover, for every Lie groupoid which integrates one of them we can find a Lie groupoid which is Morita equivalent to it and integrates the other Lie algebroid. \end{prop} 
		\begin{proof} Let $G_i$ be an integration of $A_i$. Lie's second theorem implies that a weak Morita equivalence between $A_i$ Lie algebroids over $M_i$, $i=1,2$ given by surjective submersions $q_i:S\rightarrow M_i$ induces a morphism from $(S \times_{M_i} S \rightrightarrows S) \cong \mathcal{G}(\ker Tq_i)$ to $ q_i^!(G_i)$ integrating the inclusion of Lie algebroids so Theorem \ref{liequo4} implies the result. \end{proof} 
\begin{rema} As we saw in the proof of Lemma \ref{liequo3}, if $\Phi:H=S \times_M S \rightarrow G$ integrates the inclusion $\ker Tq \hookrightarrow q^!A$ corresponding to a surjective submersion $q:S \rightarrow M$, then the $\mathtt{s}$-fibers of the quotient groupoid $\overline{G}= G/R$ are diffeomorphic to the quotients of the $\mathtt{s}$-fibers in $G$ by the restricted $H$-action, see Remark \ref{rem:embsub}. As a consequence, if the $q$-fibers are 1-connected, then $G$ is source-simply-connected if and only if so is $\overline{G}$, see \cite[Proposition 7.7]{gingro}. So the proof of Proposition \ref{morinvint} implies that a weak Morita equivalence between integrable Lie algebroids induces a Morita equivalence \cite{dufzun} between their source-simply-connected integrations, see also \cite{poista}. \end{rema}      
			A weak Morita equivalence, as proved in \cite{craruipoi}, preserves many Lie algebroid invariants, among which there are the monodromy groups which control the integrability of a Lie algebroid. However, it is not quite explicit in \cite{craruipoi} whether the uniform discreteness of these monodromy groups (which implies integrability) is preserved by such an equivalence.

\section{Integration of homogeneous spaces}\label{sec:homspa} The goal of this section is to offer some additional concrete applications of Theorem \ref{liequo4}. In this respect, homogeneous spaces in Poisson and Dirac geometry are particularly relevant since their construction in terms of infinitesimal data typically involves some sort of reduction of a Dirac structure. 

The concept of Poisson group \cite{driham} can be generalized in two different directions:
\begin{itemize} \item we can preserve the underlying group structure but generalizing the geometry over it in order to obtain the concept of ``Dirac group''. When talking about Dirac structures and Dirac morphisms we can restrict ourselves to (1) the standard Courant algebroid and forward Dirac maps or (2) we can work with Manin pairs and Manin pair morphisms. The availability of these two approaches implies that there are two possible definitions of ``Dirac group'', see \cite{ortmul,liedir}; 
\item we can generalize the underlying algebraic structure from group to groupoid while preserving the same kind of geometric structure over it (Poisson brackets): this process gives us the concept of {\em Poisson groupoid} \cite{weicoi}. \end{itemize} 
Analogously, the concepts of ``Poisson action'' and ``Poisson homogeneous space'' admit generalizations corresponding to each of the previous definitions and in all these cases there are classification results in terms of infinitesimal data. We shall study some particular examples of these constructions in which the associated Poisson or Dirac structures can be explicitly integrated thanks to Theorem \ref{liequo4}, leaving the study of other of their features and applications for later work.
\subsection{Quotients of action Lie algebroids and Dirac homogeneous spaces}\label{subsec:dirhomspa} This subsection has two main goals: we give a general criterion for the quotient of an action Lie algebroid to be integrable based on Theorem \ref{liequo4} and the construction of \cite{daz}; then we apply this criterion to the Dirac homogeneous spaces of \cite{robdir,dirhomspa} and obtain a result that generalizes the integration of Poisson homogeneous spaces in \cite{intpoihom}. 
\subsubsection{Integrability of quotients of action Lie algebroids. } Let us consider an infinitesimal action of a Lie algebra $\mathfrak{l} $ on a manifold $M$ and the associated action Lie algebroid $\mathfrak{l} \times M$. Suppose that there is a Lie subalgebra $\mathfrak{k} \hookrightarrow \mathfrak{l} $ whose restricted action is induced by the action of a Lie group $K$. If $K$ acts freely and properly on $M$ and there is an $K$-action by automorphisms on $\mathfrak{l} $ whose infinitesimal counterpart is the adjoint representation, then Lemma \ref{liequo} implies that there is a unique Lie algebroid structure on the associated bundle $A=(\mathfrak{l}/\mathfrak{k} \times M)/K$ over $M/K$ such that the quotient map $q:M \rightarrow M/K$ satisfies that $q^!A$ is isomorphic to the action Lie algebroid $\mathfrak{l} \times M$. We shall see that the explicit integration of action Lie algebroids provided in \cite{daz} implies a simple criterion for the integrability of Lie algebroids of the form of $A$; this construction generalizes the quotients of action groupoids studied in \cite{blowei}. Then we shall see some applications to Dirac geometry of this fact.

\begin{defi}[\cite{daz}] Let $\mathfrak{l} $ be a Lie algebra acting on $M$ and let $L$ be a Lie group which integrates $\mathfrak{l} $. The associated \emph{Dazord model} is the quotient groupoid $\mathcal{D}_L:= \text{hol}(\mathcal{F})/L \rightrightarrows M$ where $\mathcal{F} $ is the foliation on $M \times L$ given by the distribution $D=\{(u_M,u^l)\in TM \times TL:u\in \mathfrak{l} \}$ (here we assume that the Lie algebra structure on $\mathfrak{l} $ is given by the left-invariant vector fields). \end{defi} 
Since the action of $L$ on $P \times L$ given by $a\cdot (p,b)=(p,ab)$ preserves $D$, it induces an action by automorphisms of $L$ on $\mathcal{G} (\mathcal{F} )$, the monodromy groupoid. Since the $L$-action preserves holonomy, it descends to an action on the holonomy groupoid $\text{hol}(\mathcal{F})$; this action is free and proper and hence the quotient $\text{hol}(\mathcal{F})/L$ is a Lie groupoid which integrates the action Lie algebroid $\mathfrak{l} \times M$.
\begin{rema} In the previous definition we can use the monodromy groupoid $\text{mon}(\mathcal{F})$ instead of $\text{hol}(\mathcal{F})$. The resulting quotient, that we shall also call Dazord model, $\widetilde{\mathcal{D}}_L :=\text{mon}(\mathcal{F})/L$ is then source-simply-connected. The advantage of $\text{hol}(\mathcal{F})$ is that it is better suited for our next application. \end{rema}  
	\begin{defi}[\cite{dirhomspa}] Let $\mathfrak{l} $ be a Lie algebra and let $\mathfrak{k} \subset \mathfrak{l} $ be a Lie subalgebra. Let $K$ be a Lie group with Lie algebra $\mathfrak{k}$ and such that there is an action of $K$ on $\mathfrak{l}$ by automorphisms whose infinitesimal counterpart is the adjoint action of $\mathfrak{k} $ on $\mathfrak{l} $. In this situation $(\mathfrak{l} ,K)$ is called a {\em Harish-Chandra pair}. {\em A morphism of Harish-Chandra pairs} $(\mathfrak{l} ,K)\rightarrow (\mathfrak{l}',K')$ is a Lie group morphism $F:K\rightarrow K'$ and a Lie algebra morphism $f:\mathfrak{l} \rightarrow \mathfrak{l}'$ such that $f$ commutes with the actions and $T_eF=f|_\mathfrak{k}$.  
\end{defi}
\begin{prop}\label{daztrick} Let $(\mathfrak{l} ,K)$ be a Harish-Chandra pair and let $\mathfrak{l} \times M$ be an action Lie algebroid. Suppose that the the $\mathfrak{l} $-action on $M$ restricted to $\mathfrak{k} $ coincides with the infinitesimal action corresponding to a free and proper action of $K$ on $M$. Let $K_\circ$ be the connected component of the unit. If the inclusion of Lie algebras $\mathfrak{k} \hookrightarrow \mathfrak{l}$ is integrable by a Lie group morphism $\psi:K_{\circ} \rightarrow L$, then the associated quotient Lie algebroid $(\mathfrak{l}/\mathfrak{k}  \times M)/K$ over $M/K$ is integrable. \end{prop}  
	The proof is based on observing that, in this situation, we can promote $\psi$ to a Lie groupoid morphism $(M \times K_\circ \rightrightarrows M) \rightarrow (\mathcal{D}_L \rightrightarrows M )$ which integrates the inclusion of Lie algebroids $ \mathfrak{k} \times M \hookrightarrow \mathfrak{l} \times M$. Then (a particular case of) Theorem \ref{liequo4} implies the result.
 		
\begin{proof}[Proof of Proposition \ref{daztrick}] It is enough to prove the proposition under the assumption that $K$ is connected. If $K$ is not connected, the following discussion proves the statement for the principal bundle $K_\circ \hookrightarrow M \rightarrow M/K_\circ$. But then the discrete group $K/K_\circ$ acts by automorphisms on $(\mathfrak{l} /\mathfrak{k} \times M)/K_\circ$ and hence the quotient $(\mathfrak{l} /\mathfrak{k} \times M)/K$ is also integrable. 

Suppose that $K$ is connected and let $\psi:K \rightarrow L$ be a Lie group morphism such that $\text{Lie} (\psi)$ is the inclusion $\mathfrak{k} \hookrightarrow \mathfrak{l} $. Let us apply the Dazord construction to the restricted infinitesimal action of $\mathfrak{k} $ on $M$. We have a foliation $\widetilde{  \mathcal{F}}$ on $M \times \psi(K)$ given by the distribution $\{(u_M,u^l):u\in \mathfrak{k} \}$ and then $\text{hol}(\widetilde{  \mathcal{F}})/K$ integrates the Lie algebroid $\mathfrak{k} \times M$. We shall see that there is a morphism of Lie groupoids $\phi:\text{hol}(\widetilde{  \mathcal{F}})/K \rightarrow \text{hol}(\mathcal{F})/L$ which integrates the inclusion of Lie algebroids $\mathfrak{k} \times M \hookrightarrow \mathfrak{l} \times M$ and then we will show that $\text{hol}(\widetilde{  \mathcal{F}})/K$ is isomorphic to the action groupoid $M \times K \rightrightarrows M$. 

Consider the subbundle $D'$ of $D=\{(u_M,u^l)\in TM \times TL:u\in \mathfrak{l} \}$ given by the infinitesimal action $(u_M,u^l)$ for every $u\in\mathfrak{k}$. The foliation $\mathcal{F}'$ induced by $D'$ restricts to $\widetilde{  \mathcal{F}}$ on $M \times \psi(K) \subset M \times L$ so this inclusion induces a Lie groupoid morphism $\iota:\mathcal{G} (\mathcal{F}') \rightarrow \text{hol}(\mathcal{F})$. We claim that $\iota(a)$ is a unit for every loop $a$ with trivial holonomy. Since the foliation induced by the left (or right) cosets of $\psi(K)$ in $L$ is transversely parallelizable \cite{moeint}, a loop with trivial holonomy with respect to $\widetilde{  \mathcal{F}}$ has also trivial holonomy with respect to $  \mathcal{F}'$. Moreover, a loop with trivial holonomy with respect to $\mathcal{F}' $ has also trivial holonomy with respect to $\mathcal{F}$, see the next lemma. From these observations we deduce that $\iota$ induces a morphism $\text{hol}(\widetilde{ \mathcal{F}}) \rightarrow \text{hol}(\mathcal{F})$ which is clearly $K$-equivariant and hence it induces a Lie groupoid morphism $\phi: \text{hol}(\widetilde{ \mathcal{F}})/K \rightarrow \text{hol}(\mathcal{F})/L$. 
	
		Now $\text{hol}(\widetilde{\mathcal{F}})/K$ is isomorphic to the action groupoid $M \times K \rightrightarrows M$. In fact, any path tangent to $\widetilde{ \mathcal{F}} $ is of the form
		\[ t \mapsto (p a(t),\psi(a(t))) \] where $a$ is a path in $K$. The foliation $\mathcal{F}'$ is defined by the free action of a connected Lie group so its holonomy is trivial. Then $\text{hol}(\mathcal{F}')$ is identified with the action groupoid $(M \times \psi(K))\times K \rightrightarrows M \times \psi(K)$ corresponding to the right action $(p,a)\cdot b=(pb,a\psi(b))$. As a consequence, the quotient of $(M \times \psi(K))\times K \rightrightarrows M \times \psi(K)$ by the left action of $\psi(K)$ defined by $a\cdot (p,b,c)=(p,ab,c)$ is isomorphic to the action groupoid $M \times K \rightrightarrows M$.  
 
	Since we have a morphism $\phi: M \times K \rightarrow \text{hol}(\mathcal{F} )/L$ which integrates $\mathfrak{k} \times M \hookrightarrow \mathfrak{l} \times M$, Theorem \ref{liequo4} implies the result. \end{proof}
In order to complete the previous proof, let us prove the following. 
\begin{lem} A loop with trivial holonomy with respect to a subfoliation of $\mathcal{F} $ has also trivial holonomy with respect to $\mathcal{F} $. \end{lem} 
	\begin{proof} The (usual) proof of Frobenius theorem as in \cite{spidif} shows that if $\mathcal{F}' $ is a subfoliation of rank $p$ of $\mathcal{F} $ which has rank $q$ on a manifold $M$ of dimension $p+q+r$, we can find local coordinates $\{x^a\}_a$ around every point in $M$ such that every leaf of $\mathcal{F}' $ in that chart is given by $x^a=\text{ constant}$ for $a>p$ and the leaves of $\mathcal{F}$ are given by $x^a =\text{ constant}$ for $a>p+q$. So in a chart like this, the holonomy germ along a path $t \mapsto (a(t),x,y)\in \mathbb{R}^p\times \mathbb{R}^q \times \mathbb{R}^r $ tangent to $\mathcal{F} $ is given by the restriction to $\{(a(0),x)\}\times \mathbb{R}^r\subset \{a(0)\} \times \mathbb{R}^{q+r} $ of the holonomy germ with respect to the same path. So choosing coordinates of this kind on each open set of some finite cover of a loop $a$ tangent to $\mathcal{F}' $ with trivial holonomy we see that the holonomy of $a$ with respect to $\mathcal{F}$ is trivial too. \end{proof}

\subsubsection{Integration of Dirac homogeneous spaces.} We start by recalling some concepts in order to define the Dirac homogeneous spaces.
\begin{defi}[\cite{liuweixu}] A {\em Courant algebroid} consists of a vector bundle $E\rightarrow M$, a tensor $\langle\,,\, \rangle\in \Gamma ( E^* \otimes E^*)$ which induces a pointwise nondegenerate symmetric bilinear form that we call the {\em metric}, a bilinear bracket $\llbracket\,,\, \rrbracket:\Gamma(E)\times \Gamma(E)\rightarrow \Gamma(E)$ called the {\em Courant bracket} and a vector bundle morphism $\mathtt{a}:E\rightarrow TM$ called the {\em anchor} such that the following identities hold:
\begin{align*} &\mathtt{a}(X)\langle Y,Z \rangle =\langle \llbracket X,Y\rrbracket,Z \rangle + \langle Y, \llbracket X,Z\rrbracket \rangle, \\
& \llbracket X,\llbracket Y,Z\rrbracket\rrbracket=\llbracket\llbracket X,Y],Z\rrbracket+\llbracket Y,\llbracket X,Z\rrbracket\rrbracket,  \\
& \frac{1}{2}\mathtt{a}^*d \langle X,X \rangle = \llbracket X,X\rrbracket^\flat; \end{align*}
	for all $X,Y,Z \in \Gamma(E)$, where ${}^\flat:E\rightarrow E^*$ is the isomorphism given by the metric. \end{defi} 

\begin{defi}[\cite{liuweixu}] A subbundle $F\subset E$ of a vector bundle $E$ endowed with a metric $\langle\,,\, \rangle\in \Gamma (S^2 (E^*) )$ is called {\em isotropic} if $\langle\,,\, \rangle|_F=0$; a subbundle of $E$ with the property $E=E^\perp$ is called {\em Lagrangian}. A {\em Dirac structure} in a Courant algebroid $E$ is a Lagrangian subbundle $L\subset E$ which is involutive with respect to the restricted Courant bracket. If $L$ is a Dirac structure inside the Courant algebroid $E$, then $(E,L)$ is called a \emph{Manin pair}.
\end{defi} 
\begin{rema}[\cite{uchcou}] It is a consequence of the axioms for a Courant algebroid $(E,\langle\,,\, \rangle, \llbracket \,,\, \rrbracket, \mathtt{a} )$ that the Leibniz rule holds and $\mathtt{a}$ preserves brackets:
	\begin{align*} & \llbracket X,fY\rrbracket=f \llbracket X,Y\rrbracket+\left(\mathcal{L}_{\mathtt{a}(X)}(f)\right)Y, 
& \mathtt{a}\left(\llbracket X,Y\rrbracket\right)=[\mathtt{a}(X),\mathtt{a}(Y)]; \end{align*} 
for all $X,Y \in \Gamma(E)$ and all $f\in C^\infty(M)$. \end{rema}	
\begin{exa} A Courant algebroid over a point is a Lie algebra with an Ad-invariant metric. \end{exa}
\begin{exa} We have that $\mathbb{T}M$ endowed with the Courant-Dorfman bracket and its canonical pairing is a Courant algebroid called \emph{the canonical Courant algebroid}, see \S\ref{subsec:dirpre}. \end{exa}
\begin{exa}[\cite{liuweixu,roycou}] If $(A,\delta,\chi)$ is a Lie quasi-bialgebroid, then the bundle $A\oplus A^*$ inherits a unique Courant algebroid structure such that $A$ sits inside it as a Dirac structure, the metric is the canonical pairing and the Courant bracket and anchor restricted to $A^*$ induce the differential $\delta$. \end{exa} 
The general notion of morphism for Courant algebroids and Manin pairs is the following. Let $E,F$ be Courant algebroids over $M,N$. We denote by $\overline{F} $ the Courant algebroid $F$ with the opposite inner product. Given a smooth map $f:M \rightarrow N$, take $\Gamma_f \subset M \times N $, the graph of $f$.
\begin{defi}[\cite{alexu,buriglsev}] A {\em Courant morphism} between $E$ and $F$ over $f$ is a Lagrangian subbundle $R\subset E \times \overline{F}  $ such that: (1) the anchors satisfy $\mathtt{a}_E \times \mathtt{a}_F(R)\subset T \Gamma_f$ and (2) if $u,v\in \Gamma (E \times F)$ restrict to sections of $R$, then so does their Courant bracket $\llbracket u,v\rrbracket$. Composition of Courant morphisms $R,S$ is defined as the pointwise  composition of relations $R\circ S$. A {\em morphism of Manin pairs} $(E,L)$, $(F,K)$ over a smooth map $f: M \rightarrow N$ is a Courant morphism $R$ between $E$ and $F$ such that the composition satisfies $R\circ L=K$ and $\ker R\cap L=0$. \end{defi}  
 
Now we will review the definition of Dirac-Lie groups which generalizes the definition of Poisson groups and Lie groups endowed with the Cartan-Dirac structure \cite{purspi}. 
\begin{defi}[\cite{liedir,dirhomspa}] A {\em Dirac-Lie group} is a Manin pair $(\mathbb{A},E)$ over a Lie group $H$ equipped with a Manin pair morphism $R_{\mathtt{m}} :(\mathbb{A},E)\times (\mathbb{A},E) \rightarrow (\mathbb{A},E)$ over the multiplication map $\mathtt{m}: H \times H \rightarrow H$ such that 
	\[ R_{\mathtt{m}}\circ (R_{\mathtt{m}} \times \text{id}_{\mathbb{A}})=R_{\mathtt{m}}\circ (\text{id}_{\mathbb{A}}   \times R_{\mathtt{m}}) \quad  
	R_{\mathtt{m}}\circ (\epsilon \times \text{id}_{\mathbb{A}} )=R_{\mathtt{m}}\circ ( \text{id}_{\mathbb{A}}  \times\epsilon)\]
	where $\epsilon$ is the inclusion of the trivial Manin pair over the unit of $H$ in $(\mathbb{A},E)$. A Dirac-Lie group is called {\em exact} if its underlying Courant algebroid is exact. \end{defi}
The main result of \cite{liedir} is that Dirac-Lie groups are classified by $H$-equivariant Dirac-Manin triples, that we now recall.  
\begin{defi}[\cite{liedir}] Let $\mathfrak{d}$ be a Lie algebra and let $B\in S^2(\mathfrak{d})$ be $\mathfrak{d}$-invariant. Let $\mathfrak{g} \subset \mathfrak{d} $ be a coisotropic Lie subalgebra, i.e. $B^\sharp(\mathfrak{g}^\circ)\subset \mathfrak{g} $, where $B^\sharp: \mathfrak{d}^* \rightarrow \mathfrak{d} $ is defined by $B^\sharp(\alpha )=B(\alpha , \cdot)$ and $\mathfrak{g}^\circ \subset \mathfrak{d}^*$ is the annihilator of $\mathfrak{g} $. Then $(\mathfrak{d} ,\mathfrak{g} )$ is called a {\em Dirac-Manin pair}. \end{defi} 
In the previous definition, if $B$ is nondegenerate and $\mathfrak{g} $ is Lagrangian, we recover the definition of a Manin pair over a point.

\begin{defi} Let $(\mathfrak{d} ,\mathfrak{g} )$ be a Dirac-Manin pair and let $\mathfrak{h}\subset \mathfrak{d} $ be a Lie subalgebra such that $\mathfrak{d} =\mathfrak{g} \oplus \mathfrak{h} $ as a vector space. Then $(\mathfrak{d} ,\mathfrak{g} ,\mathfrak{h} )_B$ is called a {\em Dirac-Manin triple}. If there is an action by automorphisms of $H$ on $\mathfrak{d}$ whose infinitesimal counterpart is the action of $\mathfrak{h}$ on $\mathfrak{d}$ by inner derivations, $(\mathfrak{d} ,\mathfrak{g} ,\mathfrak{h} )_B$ is called an {\em $H$-equivariant Dirac-Manin triple}. \end{defi}

\begin{exa}\label{exa:poigro} An $H$-equivariant Dirac-Manin triple $(\mathfrak{d} ,\mathfrak{g},\mathfrak{h})_B$ in which $B$ is nondegenerate and both $\mathfrak{g} $ and $\mathfrak{h} $ are Lagrangian is called an ($H$-equivariant) Manin triple \cite{driham}. This structure induces a Poisson group structure on $H$. The corresponding Dirac-Lie group is $(H,\mathbb{T} H, \text{graph}(\pi))$, where $\pi$ is the Poisson structure that makes $H$ into a Poisson group, see \cite{meipoi}. \end{exa}  
\begin{exa}\label{cardir} Let $\mathfrak{h} $ be a Lie algebra and let $B'\in S^2(\mathfrak{h} )$ be $\mathfrak{h}$-invariant. Then the direct sum $\mathfrak{d} =\mathfrak{h} \oplus \mathfrak{h} $ with $B\in S^2(\mathfrak{d} )$ the element which restricts to $B'$ on $\mathfrak{h} \oplus 0$ and $-B'$ on $0\oplus \mathfrak{h} $ has the diagonal $\mathfrak{g} :=\mathfrak{h}_\Delta $ as a coisotropic subalgebra, so $(\mathfrak{d} ,\mathfrak{g},\mathfrak{h})_B$ is a Dirac-Manin triple, where we identify $\mathfrak{h} $ with its image in $\mathfrak{d} $ under the inclusion $u\mapsto (u,0)$. When $B'$ is a nondegenerate quadratic form we are in the situation of the {\em Cartan-Dirac structure} \cite{purspi}. For any integration $H$ of $\mathfrak{h} $, we have an $H$-action by automorphisms on $(\mathfrak{d} ,\mathfrak{g},\mathfrak{h} )_B$ which makes it into an $H$-equivariant Dirac-Manin triple. The corresponding Dirac-Lie group is isomorphic to $(H,\mathbb{T}_\eta H,E_H)$, where $\eta$ is the Cartan 3-form on $H$ and $E_H$ is the Cartan-Dirac structure, see \cite{purspi,liedir}. \end{exa} 
\begin{defi}[\cite{dirhomspa}] A {\em Dirac action} of a Dirac-Lie group $\mathcal{H} = (H,\mathbb{A},E )$ on a Manin pair $\mathcal{M} =(M,\mathbb{B},F)$ is a Manin pair morphism $R_a:\mathcal{H} \times \mathcal{M} \rightarrow \mathcal{M} $ over an action map $a:H \times M \rightarrow M$ such that  
	\[ R_a\circ (R_{\mathtt{m}}  \times \text{id}_{\mathbb{B}})=R_a\circ (\text{id}_{\mathbb{A}} \times R_a) \quad  
	R_a\circ (\epsilon \times \text{id}_{\mathbb{B}})= \text{id}_{\mathbb{B}}\]
	where $\epsilon$ is the inclusion of the trivial Manin pair over the unit of $H$ in $(\mathbb{A},E)$.
\end{defi} 
Dirac-Lie group actions on homogeneous spaces were classified in \cite{dirhomspa} as follows. Let $(\mathfrak{d} ,\mathfrak{g} ,\mathfrak{h} )_B$ be an $H$-equivariant Dirac-Manin triple. Take the following data: 
\begin{enumerate}  \item a closed Lie subgroup $K \hookrightarrow H$ and a Manin pair $(\mathfrak{n} ,\mathfrak{l})$ with bilinear form $\gamma\in S^2(\mathfrak{n} )$ such that $\mathfrak{k} \subset \mathfrak{l} $ and its $\gamma$-orthogonal  subspace $\mathfrak{k}^\perp\subset \mathfrak{n}$ is a Lie subalgebra; \item a $K$-action by automorphisms on $(\mathfrak{n} ,\mathfrak{l} )_{\gamma}$ whose infinitesimal counterpart is the action by inner derivations of $\mathfrak{k}$ on $\mathfrak{n} $; \item a morphism $(f,F):(\mathfrak{n} ,K)\rightarrow (\mathfrak{d},H)$ of Harish-Chandra pairs such that $f (\gamma)=B$. \end{enumerate} 
Define the map $\rho:\mathfrak{d} \rightarrow \mathfrak{X}(H)$ as follows: $X\in \mathfrak{d}$ goes to the vector field $X_H$ given by 
\begin{align} a\mapsto Tr_a\left( p_{\mathfrak{h} }\Ad_aX \right) \label{equ:actcoualg} \end{align}
for all $a\in H$, where $p_{\mathfrak{h} }:\mathfrak{d} \rightarrow \mathfrak{h}$ is the projection along $\mathfrak{g} $. We have that $\rho$ is a Lie algebra morphism. The morphism $\rho\circ f$ defines an action of $\mathfrak{n} $ on $H$ with coisotropic stabilizers so it induces a Manin pair structure on $(H \times \mathfrak{n} ,H \times \mathfrak{l} )$ \cite{coupoi}. Coisotropic reduction as in \cite{coupoi} of $(H \times \mathfrak{k}^\perp,H \times \mathfrak{l}) $ determines a Manin pair $(\mathbb{E} ,A)$ over $H/K$ which admits a Dirac-Lie group action of the Dirac-Lie group associated to $(\mathfrak{d} ,\mathfrak{g} ,\mathfrak{h} )_B$. In particular, $A$ is isomorphic to a quotient Lie algebroid $(\mathfrak{l}/\mathfrak{k} \times H)/K$ as in the beginning of the section. One of the main results in \cite{dirhomspa} states that all the Manin pairs with such a property are of this form. So we can call the items in 1-3, the {\em classifying data} of the associated Dirac action. Let $(\mathbb{E},A)$ be a Manin pair over $H/K$ admitting a Dirac action of a Dirac-Lie group $(H, \mathbb{A},E)$. Suppose that in the classifying data 1-3 corresponding to $(\mathbb{E},A )$ we have that $f|_{\mathfrak{l} }:\mathfrak{l} \rightarrow \mathfrak{d}$ is injective and $\mathfrak{k} =f(\mathfrak{l} )\cap \mathfrak{h} $. Then $(\mathbb{E},A)$ is a {\em Dirac homogeneous space} in the sense of \cite{robdir}, see \cite[\S 5.3]{dirhomspa}.
 
\begin{thm}\label{dirhomint} The Dirac homogeneous spaces are integrable. \end{thm} 
	\begin{proof} It is enough to prove the result for Dirac homogeneous spaces over connected homogeneous spaces, see \cite[Thm. 4.4]{intpoihom} and its proof. Let $(\mathbb{E},A)$ be a Dirac homogeneous space over $H/K$, where $(H, \mathbb{A},E)$ is a connected Dirac-Lie group. Let $(\mathfrak{d} ,\mathfrak{g} ,\mathfrak{h} )_B$ be the $H$-equivariant triple associated to $(H, \mathbb{A},E)$. Since there is also a Dirac-Lie group structure over the 1-connected covering $\widetilde{H}$ of $H$ and a canonical Dirac-Lie group morphism over the covering map $\widetilde{H} \rightarrow H$, we can assume, without loss of generality, that $H$ is 1-connected. So there is a Lie group morphism $j:H \rightarrow D$ integrating the inclusion $\mathfrak{h} \hookrightarrow \mathfrak{d} $. Consider the classifying data 1-3 associated to $(\mathbb{E},A)$. Since $f|_:\mathfrak{l} \rightarrow \mathfrak{d} $ is injective, there is an immersed connected subgroup $L\subset D$ with Lie algebra $\mathfrak{l} $. Then $j(K_\circ)\subset L$, where $K_\circ\subset K$ is the connected component of the unit. So the inclusion $\mathfrak{k} \hookrightarrow \mathfrak{l} $ is integrable by $\psi:=j|_{K_\circ}:K_\circ \rightarrow L$. Therefore, Proposition \ref{daztrick} implies the result. \end{proof}
		This is a generalization of the main result of \cite{intpoihom} which concerns Poisson homogeneous spaces of Poisson groups. In fact, more generally, we have the following.
\begin{coro} The Dirac structures associated to Dirac actions of exact Dirac-Lie groups on exact Manin pairs over homogeneous spaces are integrable. \end{coro} 
\begin{proof} In this situation, the classifying data 1-3 of the Dirac action correspond to a Harish-Chandra pair morphism which is actually an isomorphism, see \cite[Proposition 5.10]{dirhomspa}, so the associated data reduces to a Lagrangian subalgebra of $\mathfrak{d}$ and hence Theorem \ref{dirhomint} implies the result. \end{proof}
Poisson groups and the Cartan-Dirac structure provide examples of exact Dirac-Lie groups so this result already covers the situations more commonly studied in the literature. Dirac homogeneous spaces \cite{jotdir} corresponding to Dirac-Lie groups in the sense of \cite{ortmul} are also integrable, see \cite{intquopoi}. 

\subsection{Poisson homogeneous spaces of symplectic groupoids and Poisson groups}
 
\begin{defi}[\cite{weicoi}] A {\em Poisson groupoid} is a Lie groupoid $\mathcal{G}  \rightrightarrows M$ with a Poisson structure on $\mathcal{G}$ such that the graph of the multiplication map is a coisotropic submanifold\footnote{Let $(M,\pi)$ be a Poisson manifold. A submanifold $C$ of $M$ is coisotropic if $\pi^\sharp(T^\circ C)\subset TC$, where $T^\circ C$ is the annihilator of $TC$. } of $\mathcal{G}\times \mathcal{G}\times \overline{\mathcal{G}}$, where $\overline{\mathcal{G}}$ denotes $\mathcal{G}$ with the opposite Poisson structure. \end{defi}
\begin{exa} A Poisson groupoid over a point is a {\em Poisson group} \cite{driham}. \end{exa}
\begin{exa} A Poisson groupoid whose bracket is nondegenerate is a symplectic groupoid. \end{exa}
The infinitesimal description of Poisson groupoids is provided by Lie bialgebroids, see \cite{macxu}. If $(A,A^*)$ is a Lie bialgebroid over $M$, we have that the map $\pi^\sharp :=\mathtt{a}\circ \mathtt{a}_*^*:T^*M \rightarrow TM$ defines a Poisson structure $\pi$ on $M$, where $\mathtt{a}$ and $\mathtt{a}_*$ are the anchors of $A$ and $A^*$ respectively; if $(A,A^*)$ is integrable by a Poisson groupoid $\mathcal{G} \rightrightarrows M$, then $\pi$ is the unique Poisson structure on $M$ that makes $\mathtt{t}:\mathcal{G} \rightarrow M$ into a Poisson morphism. If $(A,A^*)$ is the tangent Lie bialgebroid of a symplectic groupoid $\mathcal{G} \rightrightarrows M$, then $A$ is isomorphic to the cotangent Lie algebroid $T^*M$ \S \ref{subsec:poi} and $A^*$ is isomorphic to $TM$.      
	
\begin{defi}\cite{liuweixu2} Let $\mathcal{G} \rightrightarrows M$ be a Poisson groupoid. A {\em Poisson action} of $\mathcal{G}$ on a Poisson manifold $P$ is a groupoid action $a:{\mathcal{G}} \times_M P \rightarrow P$ of $\mathcal{G}$ on $J:P \rightarrow M$ such that its graph is coisotropic in $\mathcal{G} \times P \times \overline{P} $. If there is a section $s:M \rightarrow P$ of $J$ such that $P=\mathcal{G}  \cdot s(M)$ for the induced action of ${\mathcal{G}}$, $P$ is called a {\em homogeneous space} of $\mathcal{G} $.
\end{defi}
Poisson groups are always integrable \cite{luphd} so they allow us to consider two kinds of Poisson homogeneous spaces:
\begin{itemize} \item Poisson manifolds endowed with a transitive Poisson action of a Poisson group as in the original treatment \cite{dripoi}.
\item Poisson manifolds endowed with a transitive Poisson action of some symplectic groupoid which integrates a Poisson group. \end{itemize}  

Let us explain the second situation in more familiar terms. The Lie functor establishes an equivalence of categories between the category of simply-connected Poisson groups and the category of Lie bialgebras \cite{driham,driqua}. Hence, to every Poisson group $G$ there is an associated 1-connected Poisson group $G^*$ with tangent Lie algebra $\mathfrak{g}^*$. A Poisson action of a Poisson group $G$ on a manifold $M$ is {\em Hamiltonian} if it is encoded infinitesimally by a Poisson morphism $J: M \rightarrow G^*$, see \cite{luphd}. This concept generalizes the classical moment maps with target the dual of the Lie algebra $\mathfrak{g}^*$. A {\em Hamiltonian Poisson action of $G$} induces an action of the source-simply-connected symplectic groupoid $\Sigma(G^*) \rightrightarrows G^*$ which integrates $G^*$; if such an action is transitive along the fibers of the moment map $J$, then $M$ is a Poisson homogeneous space of $\Sigma(G^*) \rightrightarrows G^*$.

There are several explicit examples and applications of the original Poisson homogeneous spaces of Poisson groups, see \cite{dripoi,intpoihom} and references therein, while there are few nontrivial examples of Hamiltonian Poisson actions and even fewer that come with a transitive symplectic groupoid action in the previous sense. 

In this subsection we shall see the following topics: (i) as a corollary of Theorem \ref{liequo4} we shall produce a general criterion for the Poisson homogeneous space of a symplectic groupoid to be integrable; (ii) based on the general infinitesimal classification of Poisson homogeneous spaces provided in \cite{liuweixu2}, for every Poisson homogeneous space of a (connected) Poisson group $G$ we shall construct a Poisson homogeneous space of some symplectic groupoid $\mathcal{G} \rightrightarrows G^*$ with the same underlying infinitesimal data; (iii) we shall prove that the Poisson homogeneous spaces constructed in the previous way are explicitly integrable thanks to our general criterion.  
\subsubsection{Integrability of Poisson homogeneous spaces of symplectic groupoids. } Let us recall the classification of Poisson homogeneous spaces for Poisson groupoids of \cite{liuweixu2}. For any Lie groupoid $\mathcal{G}  \rightrightarrows M$ with Lie algebroid $A$, there is a canonical Lie groupoid structure on $T^* \mathcal{G} $ over $A^*$ \cite{macgen}. If $\mathcal{G}  $ is a Poisson groupoid, the target map $\hat{\mathtt{t}}:T^*\mathcal{G} \rightarrow A^*$ is a Lie bialgebroid morphism \cite{macgen} so it induces a Courant morphism $R:\mathbb{T}  \mathcal{G} \rightarrow A\oplus A^*$ \cite{alexu}. Take a Dirac structure $L \subset A\oplus A^*$. Then its pullback $E:=L\circ R$ is 
\begin{align} E=\{X^r+\pi_\mathcal{G}^\sharp(\alpha )\oplus \alpha : X \oplus\hat{ \mathtt{t}} (\alpha) \in L\}.\label{pulldir} \end{align}  
If $L \cap (A\oplus 0)$ is of constant rank and there is a Lie subgroupoid $\mathcal{K}\subset \mathcal{G}$ whose Lie algebroid is $L \cap (A\oplus 0)$ and is suitably regular, then $E$ is reducible and its Poison quotient is identified with the quotient by left translations $\mathcal{G}/\mathcal{K}$. The main result of \cite{liuweixu2} states that all the Poisson homogeneous spaces of $G$ are of this form. 
\begin{exa} Taking $L=A\oplus 0$, the Poisson homogeneous space associated to $L$ is $\mathcal{G}/\mathcal{G}=M$ itself. The Poisson structure on $M$ is integrable if so is $\mathcal{G}$, see \cite{morla}. \end{exa}
If $\mathcal{G}  \rightrightarrows M$ is a symplectic groupoid, then $E$ can be described in more familiar terms thanks to the fact that the Courant algebroid emerging as the double of its Lie bialgebroid is isomorphic to $\mathbb{T}M$ by means of the map $e^\pi: \mathbb{T}M \rightarrow \mathbb{T}M$: $X \oplus \alpha \mapsto X+\pi(\alpha )\oplus \alpha $, where $\pi$ is the Poisson bracket on $M$ \cite{liuweixu,liuweixu2}.   
\begin{lem}\label{poihom1} Let $L\subset \mathbb{T} M$ be a Dirac structure on a Poisson manifold $(M,\pi)$. Suppose that $M$ is integrable by the symplectic groupoid $({\mathcal{G}},\omega )$. Then $E$ as in \eqref{pulldir} is isomorphic as a Lie algebroid to the backward image $\mathfrak{B}_\mathtt{t}( L)$. \end{lem}
	\begin{proof} The map $\omega^\flat:T \mathcal{G} \rightarrow T^*{\mathcal{G}}$ is an isomorphism of Lie groupoids and hence the target map of the cotangent groupoid, $\hat{ \mathtt{t}}:T^*{\mathcal{G}} \rightarrow TM$, is identified with $T\mathtt{t}\circ( \omega^\flat)^{-1}$. The Poisson groupoid structure on $\mathcal{G} $ is given by $\Pi:= \omega^{-1}$ which satisfies $\Pi^\sharp=(\omega^\flat)^{-1}$. Since $\Pi^\sharp(\mathtt{t}^*\theta)=\theta^r$ and $\mathtt{t}$ is a Poisson morphism we have that 
	\[ e^{-\omega}(E)= \{\Pi^\sharp(\alpha +\mathtt{t}^*\theta )\oplus-\mathtt{t}^*\theta :T \mathtt{t}\Pi^\sharp(\alpha )\oplus \theta \in e^\pi L \}=\mathfrak{B}_\mathtt{t}(L^{op});\] where $L^{op}=\{X\oplus \alpha |X\oplus -\alpha \in L\}$ which is isomorphic to $L$. Therefore, we have the result. \end{proof}
\begin{prop}\label{poihom2} Let $P=\mathcal{G} /\mathcal{K}$ be a Poisson homogeneous space of a symplectic groupoid $\mathcal{G}\rightrightarrows M$ and $L\subset \mathbb{T} M$ its associated Dirac structure, where $\text{Lie}(\mathcal{K})=L \cap T^*M$. Then $P$ is integrable if the inclusion of Lie algebroids $L\cap T^*M \hookrightarrow L$ is integrable by a Lie groupoid morphism with source $\mathcal{K}$.
\end{prop}
  
\begin{proof} Suppose that the inclusion of Lie algebroids $L\cap T^*M \hookrightarrow L$ is integrable by a Lie groupoid morphism $\psi:\mathcal{K} \rightarrow \mathcal{L}$. Lemma \ref{poihom1} implies that $E\cong \mathtt{t}^!_{\mathcal{G} } L$ so $E$ is integrable by the pullback groupoid 
\[ \mathtt{t}^!_\mathcal{G} \mathcal{L} =\{(x,k,y)\in \mathcal{G} \times \mathcal{L} \times \mathcal{G} :\mathtt{t}_{\mathcal{G} }(x)=\mathtt{t}_\mathcal{L}(k),\mathtt{t}_{\mathcal{G} }(y)=\mathtt{s}_\mathcal{L}(k)\}. \]
Now $P$ is the orbit space of the action groupoid $\mathcal{X}: =\mathcal{K}\times_M \mathcal{G}  \rightrightarrows \mathcal{G}$ and we can define $\Psi: \mathcal{X} \rightarrow \mathtt{t}^!_\mathcal{G} \mathcal{L} $ as
	\begin{align}   (k,x)\mapsto (kx,\psi(k),x).\label{grmp} \end{align}  
		Since $\Psi$ integrates the inclusion $\text{Lie}(\mathcal{X} ) \hookrightarrow E$ and $\mathcal{X}$ is isomorphic to the submersion groupoid $ \mathcal{G} \times_P \mathcal{G} \rightrightarrows \mathcal{G} $, Theorem \ref{liequo4} implies that the Poisson structure on $P$ is integrable. \end{proof}

\begin{rema}\label{morequpoihomspa} In this situation, the pullback groupoid $\mathtt{t}^!_{\mathcal{G} } \mathcal{L} $ which is Morita equivalent to $\mathcal{L}$. On the other hand, the Poisson structure $\pi$ on $P$, the Poisson homogeneous space associated to $L$, is integrable by a quotient of $\mathtt{t}^!_{\mathcal{G} }\mathcal{L} $ as in Theorem \ref{liequo4} which is also Morita equivalent to $\mathtt{t}^!_{\mathcal{G} }\mathcal{L} $. Therefore, $L$ and the cotangent Lie algebroid $T^*P$ associated to $\pi$ are integrable by Morita equivalent Lie groupoids. \end{rema}
\subsubsection{Poisson homogeneous spaces of Poisson groups and their Hamiltonian counterparts. } Let $G$ be a Poisson group and let $\mathfrak{d} $ be the double of the tangent Lie bialgebra of $G$. We have an action Courant algebroid $\mathfrak{d} \times G$ given by the map $\rho:\mathfrak{d} \rightarrow \mathfrak{X}(G)$ defined by \eqref{equ:actcoualg} (putting $H=G$), see \cite{coupoi}. The right-invariant trivialization $u\oplus \theta\mapsto u^r+\pi(\theta^r)\oplus \theta^r$ of $\mathbb{T} G$ defines an isomorphism with the action Courant algebroid $\mathfrak{d} \times G$. In what follows we shall identify $\mathbb{T}G$ with $\mathfrak{d} \times G$ by means of this isomorphism. Let $ \mathfrak{l} \times G \hookrightarrow \mathfrak{d} \times G$ be the Dirac structure associated to a Lagrangian subalgebra $\mathfrak{l} \subset \mathfrak{d} $. The Lie algebroid bracket $[\,,\,]$ and the Courant bracket $\llbracket\,,\, \rrbracket$ on $\mathfrak{d} \times G$ are related as follows:
		\[ \llbracket u,v \rrbracket=[u,v]+\mathtt{a}^* \langle du,v \rangle, \] 
		where $u,v\in C^\infty(G,\mathfrak{d} )$, see \cite[Lemma 4.1]{coupoi}. So if $u$ and $v$ take values in a Lagrangian subalgebra, then we have $\llbracket u,v \rrbracket=[u,v]$. 
	 
\begin{defi}[\cite{intpoihom}] Let $G$ be a Poisson group. A {\em Drinfeld double} of $G$ is an integration $D$ of the double Lie algebra $\mathfrak{d} $ such that there is a Poisson group morphism $G \rightarrow D$ which integrates the inclusion of Lie algebras $\mathfrak{g} \hookrightarrow \mathfrak{d} $. \end{defi} If $G$ is 1-connected, then it automatically admits a Drinfeld double given by the 1-connected integration of $\mathfrak{d} $. Throughout this section we shall use the following notation. If there are group morphisms $G \rightarrow  D$, $G^* \rightarrow  D$ which integrate the inclusions of Lie algebras $\mathfrak{g} \hookrightarrow \mathfrak{d} $, $\mathfrak{g}^* \hookrightarrow \mathfrak{d} $, we denote the image of an element $x\in G$ in $D$ as $\overline{x}$. Let $G$ be a Poisson group which admits a Drinfeld double $D$ and let us denote by ${G^*} $ the image in $D$ of the 1-connected integration of $\mathfrak{g}^*$. 

Let us recall the integration of the Poisson structures on Poisson groups given in \cite{luwei2,luphd}. There is a symplectic groupoid which integrates the Poisson structure on $G$. We take 
		\begin{align} \mathcal{G} =\{(g,u,v,h)\in G \times {G^*} \times {G^* }\times G: \overline{g}{u}={v}\overline{h}\}\label{sympoigr}, \end{align} 
The source and target of this groupoid are the projections to $G$ and the multiplication is given by \eqref{comp1}.
	\begin{align}\label{comp1}  \mathtt{m}^G( (a,u,v,b),(b,u',v',c))=(a,uu',vv',c); \end{align}

 In fact, $\mathcal{G} $ is also an integration for the Poisson structure on $G^*$: its source and target maps are the projections to $G^*$ and the multiplication is analogous to \eqref{comp1}, see \cite{luphd}. 

Take a Lagrangian subalgebra $\mathfrak{l} \subset \mathfrak{d} $. The general classification of \cite{liuweixu2} allows us to construct four associated (not necessarily smooth) homogeneous spaces in the following way. Let us denote $\mathfrak{k}=\mathfrak{g}^* \cap \mathfrak{l} $, $\mathfrak{k}'=\mathfrak{g} \cap \mathfrak{l} $ and let $K \subset {G^*}$, $K'\subset G$ be the connected subgroups integrating $\mathfrak{k}, \mathfrak{k}'$ respectively. In this situation we can integrate $\mathfrak{l}\times G $ with the following Lie groupoid which is an adaptation of \eqref{sympoigr} and was used in \cite{intpoihom} to integrate the Poisson homogeneous spaces of Poisson groups:
		\begin{align}  \mathcal{L}= \{(g,x,u,h)\in G \times L \times { G^*} \times G: \overline{g}{x} ={u}\overline{h}\} \label{laggro}\end{align}  
			where $L\subset D$ is the connected subgroup which integrates $ \mathfrak{l}  \subset \mathfrak{d} $ and $a \mapsto \overline{a}$ denotes again the inclusion of the corresponding element in $D$. The structure maps of $\mathcal{L} $ are analogous to those of $\mathcal{G} $, for example, the multiplication is $\mathtt{m}((g,x,u,h),(h,y,v,k))=(g,xy,uv,k)$. Similarly, we have an integration $\mathcal{L}'$ of the Dirac structure $\mathfrak{l} \times G^*$. The Lie algebroid $\mathfrak{k} \times G$ is integrable by
	\begin{align} 	 \mathcal{K}= \{(g,x,u,h)\in G \times K \times {G^*} \times G: \overline{g}{x}={u}\overline{h}\} \rightrightarrows G \label{isogro}\end{align} 
		with structure maps as those of $\mathcal{L} $. We have that $\mathcal{K}$ is a subgroupoid of both ${\mathcal{G}}$ and $\mathcal{L} $. Analogously, $\mathfrak{k}' \times G^*$ can be integrated by a Lie subgroupoid $\mathcal{K}'\subset \mathcal{L}'$ defined in terms of $K'$. We can ask what are the conditions that ensure the smoothness of the following topological spaces (in which case they automatically become Poisson homogeneous spaces):
\[ G/K', \quad G^*/K, \quad \mathcal{G}/\mathcal{K}, \quad \mathcal{G} /\mathcal{K}'. \]       	
Notice that $\mathcal{G} /\mathcal{K} $	is a homogeneous space for $\mathcal{G} \rightrightarrows G$ while $\mathcal{G} /\mathcal{K}'$ is a homogeneous space for $\mathcal{G} \rightrightarrows G^*$.			
		\begin{prop}\label{bijpoihomspa} We have that ${\mathcal{G}} /\mathcal{K}$ is a manifold if and only if so is ${G^*}/K$. \end{prop} 
\begin{proof} Suppose that $K$ is closed in $\overline{G^*}$. Then $\mathcal{K} \subset \mathcal{G} $ is a closed submanifold. The action groupoid $\mathcal{X}: =\mathcal{K} \times_G \mathcal{G} \rightrightarrows \mathcal{G} $ which defines the orbit space $\mathcal{G} /\mathcal{K} $ is a proper Hausdorff Lie groupoid with trivial isotropy groups. In fact, the full action groupoid $\mathcal{S}:= \mathcal{G} \times_G \mathcal{G} \rightrightarrows \mathcal{G} $ is proper. Take a product of compact subsets $X \times Y \subset \mathcal{G} \times \mathcal{G} $, then $Z=(X \times \mathtt{i}_\mathcal{G}  (Y))\cap \mathcal{G} \times_G \mathcal{G}$ is also compact and so is $Z'=(\mathtt{t}_\mathcal{S},\mathtt{s}_\mathcal{S} )(Z)$. But $Z'$ is homeomorphic to $(\mathtt{t}_\mathcal{S},\mathtt{s}_\mathcal{S} )^{-1}(X \times Y)$: if $(a,b)\in (\mathtt{t}_\mathcal{S},\mathtt{s}_\mathcal{S} )^{-1}(X \times Y)$, then $(\mathtt{m}_{\mathcal{G} }( a,b),\mathtt{i}_{\mathcal{G} }( b))\in Z$ and hence $(a,\mathtt{i}_{\mathcal{G} }( b))\in Z'$. Since $\mathcal{K} \subset \mathcal{G} $ is closed, $\mathcal{X} \rightrightarrows \mathcal{G} $ is also proper. As a consequence, the equivalence relation $R=(\mathtt{t}_\mathcal{X},\mathtt{s}_\mathcal{X} )(\mathcal{K} \times_G \mathcal{G} )\subset \mathcal{G} \times \mathcal{G} $ which defines the quotient ${\mathcal{G}} /\mathcal{K}$ is the image of an injective proper immersion and so it is a closed embedded submanifold such that $\text{pr}_2: \mathcal{G} \times \mathcal{G}\rightarrow \mathcal{G} $ restricted to $R$ is a submersion. Therefore, Godement's Theorem implies that ${\mathcal{G} }/\mathcal{K}$ is a Hausdorff manifold. 

Conversely, suppose that $P={\mathcal{G}} /\mathcal{K}$ is smooth. So the equivalence classes of $R$ are closed. Since $S=\{(1,x,x,1)\in \mathcal{G} :\forall x \in K\}$ lies inside an equivalence class of $R$ it is also closed. Therefore, $K$ is closed in ${G^*}$ as well. \end{proof} 
		We know that $G^*/K$ is integrable provided it is smooth, see \cite{intpoihom} or \S \ref{subsec:dirhomspa}. We observe here that also $\mathcal{G}/\mathcal{K} $ is integrable.

\begin{coro}\label{symgrohampoihomspa} The Poisson manifolds $\mathcal{G} /\mathcal{K}$ of Proposition \ref{bijpoihomspa} are integrable. \end{coro}
\begin{proof} The inclusion of Lie algebroids $\mathfrak{k} \times G \hookrightarrow \mathfrak{l} \times G $ is integrable by the inclusion $\mathcal{K} \hookrightarrow \mathcal{L} $ so Proposition \ref{poihom2} applied to the symplectic groupoid $\mathcal{G} $ implies that $P=\mathcal{G} /\mathcal{K} $ is integrable. More concretely, the Poisson structure on $P$ is the reduction of a Dirac structure $E$ of the form \eqref{pulldir} which is isomorphic as a Lie algebroid to the pullback Dirac structure $\mathtt{t}^!_{\mathcal{G} }(\mathfrak{l}\times G)$ on $\mathcal{G} $, where $\mathtt{t}_{\mathcal{G} }:\mathcal{G} \rightarrow G$ is the target map, see Lemma \ref{poihom1}. Then $E\cong \mathtt{t}^!_{\mathcal{G} }(\mathfrak{l}\times G)$ is integrable by the pullback Lie groupoid $\mathtt{t}^!_{\mathcal{G} }\mathcal{L} $ and, by Theorem \ref{liequo4}, the quotient Poisson structure on $P$ is integrable by the orbit space $\mathtt{t}^!_{\mathcal{G} } \mathcal{L}/ R$, where $R$ is the equivalence relation as in \eqref{equrel} associated to the Lie groupoid morphism \eqref{grmp}. \end{proof}
To conclude we shall see a couple of simple families of Poisson homogeneous spaces of Poisson groups, their Hamiltonian counterparts and their integrations according to Corollary \ref{symgrohampoihomspa}.  
		\begin{exa}\label{exa:poihomspa1} Let $G$ be a 1-connected complete Poisson group and let $G^*$ be the 1-connected integration of $\mathfrak{g}^*$. In this situation, the dressing actions define a Lie group structure on $D=G \times G^*$ which integrates $\mathfrak{d} $ and so $D$ is a Drinfeld double for $G$ \cite[Proposition 2.43]{luphd}. As a consequence, $\Sigma(G) \rightrightarrows G$ is an action groupoid $G \times G^* \rightrightarrows G$ which is isomorphic to $\mathcal{G} \rightrightarrows G $ as in \eqref{sympoigr}. Let $\mathfrak{l}\subset \mathfrak{d}$ be a Lagrangian subalgebra which corresponds to a Poisson homogeneous space of $G^*$. Then $\mathcal{K} $ as in \eqref{isogro} is isomorphic to an action groupoid as well $G \times K \rightrightarrows G$, where $K\subset G^*$ is the connected integration of $\mathfrak{k}$. Therefore, we have the following: the Poisson homogeneous space $P=\mathcal{G} /\mathcal{K} $ is diffeomorphic to the quotient $(G \times G^* )/K$, where the $K$-action is defined as $x\cdot (a,u)=(a\cdot x^{-1},xu)$ for all $(a,u)\in G \times G^*$ and all $x\in K$. So $P$ is a fiber bundle over $G^*/K$ with typical fiber $G$: 
\[ G \hookrightarrow P \rightarrow G^*/K. \] 
Similarly, the integration of $P$ that we get by applying the argument in Proposition \ref{poihom2} is a fiber bundle construction. We have that $\mathcal{L} $ as in \eqref{laggro} is an action groupoid $G \times L \rightrightarrows G$. The pullback groupoid $\mathtt{t}^!_{\mathcal{G} } \mathcal{L}$ is isomorphic to the product groupoid of $\mathcal{L} $ and the pair groupoid $G^* \times G^* \rightrightarrows G^*$. Now $K \times K$ acts on $\mathcal{L}$ by means of the action map $(x,y)\cdot (a,l)= (a\cdot x^{-1},xly^{-1})$ for all $(x,y,a,l)\in K \times K \times G \times L$; on the other hand, we have a principal $K \times K$-action on $G^* \times G^*$ defined by $(x,y)\cdot(u,v)=(xu,yv)$ for all $(x,y,u,v)\in K \times K \times G^* \times G^*$. Both actions are multiplicative in the sense that they are groupoid morphisms:
\begin{align*} (K \times K )\times \mathcal{L} &\rightarrow \mathcal{L}, \\
   (K \times K) \times (G^* \times G^*) &\rightarrow G^* \times G^*, \end{align*}
with respect to the pair groupoid structure $K \times K \rightrightarrows K$. More precisely, they are Lie 2-group actions. So the quotient $(G^* \times G^*)/(K \times K)$ is isomorphic to the pair groupoid over $G^*/K$ and the Lie groupoid $\mathcal{P} $ integrating $P$ that the proof of Proposition \ref{poihom2} gives us is a fibration of Lie groupoids:
\[ \mathcal{L} \hookrightarrow \mathcal{P} \rightarrow(G^*/K)\times(G^*/K).\] \end{exa}  

The following simple family of examples shows that there is not an obvious relationship such as (weak) Morita equivalence between the Poisson structure on $P=\mathcal{G} /\mathcal{K} $ and the Poisson structure on $G^*/K$.
\begin{defi}[\cite{dazson}] Let $(G,\pi)$ be a Poisson group. An {\em affine Poisson structure} on $G$ is a Poisson structure $\Pi$ on $G$ such that	the multiplication map $\mathtt{m}: (G,\pi) \times (G,\Pi) \rightarrow (G,\Pi)$ is a Poisson map. \end{defi} 
		Affine Poisson structures are the same as Poisson homogeneous space structures on $G$. In particular, these structures are classified by Lagrangian subalgebras $\mathfrak{l}\subset \mathfrak{d} $ such that $\mathfrak{l} \cap \mathfrak{g} =\{0\}$.

\begin{exa} Take a Lagrangian subalgebra $\mathfrak{l}\subset \mathfrak{d} $ such that $\mathfrak{l} \cap \mathfrak{g}^* =\{0\}$. In this situation, the classification theorem of \cite{liuweixu2} implies that we can take $\mathcal{G} :=\Sigma(G)$ as the underlying manifold of a Poisson homogeneous space associated to $\mathfrak{l}  \times G$ (provided that it is Hausdorff). Since $\mathfrak{l} \times G$ is isomorphic as a Lie algebroid to the cotangent Lie algebroid $\mathfrak{g}^* \times G\cong T^*G$ \cite[Ch. 5]{luphd}, we can integrate $\mathfrak{l} \times G$ with the Lie groupoid $\mathcal{G}\rightrightarrows G $ itself. But then the pullback groupoid $\mathtt{t}^!_{\mathcal{G} } \mathcal{G} $, which integrates $\mathtt{t}_{\mathcal{G} }^!(\mathfrak{l} \times G)\cong E$ as in \eqref{pulldir}, is isomorphic to the pair groupoid $\mathcal{G} \times \mathcal{G} \rightrightarrows \mathcal{G} $. Therefore, $\mathcal{G} \times \mathcal{G} \rightrightarrows \mathcal{G} $ is already a symplectic groupoid which means that $E$ is the graph of a symplectic structure. So we see that under the correspondence of Proposition \ref{bijpoihomspa} an affine Poisson group of $G^*$ corresponds to what might be called an ``affine symplectic structure'' on $\mathcal{G} $, since it is determined by a symplectic structure on $\mathcal{G} $ such that the action by right multiplication of $\mathcal{G} $ on itself is a symplectic action. See \cite{burdru} for the definition of affine tensors on Lie groupoids. \end{exa}
	\subsubsection{Integration in the case of a connected Poisson group.} Corollary \ref{symgrohampoihomspa} is based on the assumption that $G$ admits a Drinfeld double. In the next paragraph we shall adapt this integration result to connected Poisson groups at the expense of requiring an additional hypothesis on the source-simply-connected integration of the Poisson structure. Suppose that $G$ is a connected Poisson group. Then there is a surjective Poisson group morphism $q:\widetilde{G}\rightarrow G$ where $\widetilde{G}$ is a 1-connected Poisson group with the same Lie bialgebra. Since the Poisson tensor of $G$ vanishes on 1, the Poisson tensor of $\widetilde{G}$ vanishes on $Z:=\ker q$. As we have said, $\widetilde{G}$ admits a Drinfeld double so we can define the corresponding Lie groupoids $\mathcal{G} $, $\mathcal{K} $, $\mathcal{L} $.  
		\begin{lem} We have that $Z$ acts on ${\mathcal{G}}$ and on $\mathcal{L} $ by automorphisms and the map $\Phi:\mathcal{K} \rightarrow \mathcal{L} $ given by $\Phi(a,x,u,b)=(a,i(x),u,b)$ is $Z$-equivariant, where $i:K \rightarrow L$ is the inclusion. \end{lem} 
\begin{proof} The Poisson tensor $\pi$ on $\widetilde{G}$ can be expressed in terms of the inclusions into $D$ and the projections $\text{pr}_1: \mathfrak{d} \rightarrow \mathfrak{g} $, $\text{pr}_2: \mathfrak{d} \rightarrow \mathfrak{g}^*$ as follows, see \cite[Proposition 2.31]{luphd}:
		\[ (r_{a^{-1}}\pi_a)(\alpha, \beta )=-\langle \text{pr}_1 \Ad_{\overline{a}^{-1}} \alpha , \text{pr}_2 \Ad_{\overline{a}^{-1}} \beta \rangle \] 
for all $\alpha, \beta \in \mathfrak{g}^*$ and all $a\in G$. Using the explicit description of the Lie bracket on $\mathfrak{d}$ we get that  $\text{pr}_2 \Ad_{\overline{a}^{-1}} \beta=\Ad^*_a \beta $, where $\Ad^*$ is the coadjoint action of $\widetilde{G}$ on $\mathfrak{g}^*$. So we have that 
		\[ \langle \text{pr}_1 \Ad_{\overline{a}^{-1}} \alpha ,  \Ad^*_{{a}} \beta \rangle= 0 \]
for all $\alpha, \beta \in \mathfrak{g}^*$ and all $a\in Z$. As a consequence, for all $\alpha  $ we have that $\text{pr}_1 \Ad_{\overline{a}^{-1}} \alpha=0$ and hence the adjoint action of $Z$ on $\mathfrak{d}$ fixes $\mathfrak{g}^*$. Therefore, for all $z\in Z$ and all $v\in {G^*}$ there exists $w\in {G^*}$ such that $\overline{z}{v}= {w}\overline{z} $, in such a situation we denote $w={}^zv$. Let us define an action of $Z$ on $\mathcal{G}$ as follows:
		\[ z\cdot (a,u,v,b)=(za,u,{}^zv,zb). \]

This is a free and proper action by automorphisms whose orbit space is a Lie groupoid which integrates the Poisson structure on $G$.

We can also define an action by automorphisms of $Z$ on $\mathcal{L} $ as in \eqref{laggro} by means of the same formula: $z\cdot (a,x,u,b)=(za,x,{}^zu,zb)$ for all $z\in Z$ and $(a,x,u,b)\in \mathcal{L} $. So the map $\Phi$ as before is $Z$-equivariant. \end{proof}  
\begin{thm}\label{poihomspasymgro} Let $G$ be a connected Poisson group and let $P$ be a Poisson homogeneous space of the form $\Sigma(G)/\mathcal{K}'$, where $\mathcal{K}'\rightrightarrows G $ is source-connected. Suppose that the Dirac structure associated to $P$ is of the form $\mathfrak{l} \times G$ under the left trivialization $\mathbb{T}G\cong \mathfrak{g} \times G $, where $\mathfrak{l} $ is a Lagrangian subalgebra of the double $ \mathfrak{d}  $. Then $P$ is integrable. \end{thm} 
%Since the proof of this theorem relies on Proposition \ref{poihom2}, we have to require $\Sigma(G) \rightrightarrows G$ to be Hausdorff, see Remark \ref{haucon}. 
	\begin{proof} Let $P$ be a Poisson homogeneous space of the form $\Sigma(G)/\mathcal{K}'$ whose associated Dirac structure is isomorphic to $\mathfrak{l} \times G \hookrightarrow \mathfrak{d} \times G$ via left trivialization. By Lie's second theorem there is a Lie groupoid morphism $\Psi:\Sigma(G) \rightarrow \mathcal{G} /Z$ which covers the identity on $T^*G\cong \mathfrak{g}^* \times G$. Since $\mathcal{K}'$ is source-connected and $\text{Lie} (\mathcal{K}')= \text{Lie} (\mathcal{K} /Z)\cong(\mathfrak{l} \cap \mathfrak{g}^*) \times G$, we have that $\Psi(\mathcal{K}')\subset \mathcal{K}/Z$. So we can compose $\Psi|_{\mathcal{K}'} $ with the morphism $\overline{\Phi} : \mathcal{K}/Z \rightarrow \mathcal{L}/Z$ induced by $\Phi$ as in the previous lemma and satisfy the condition of Proposition \ref{poihom2}. As a consequence, $P$ is integrable. \end{proof}

\subsection*{Acknowledgments.} The author thanks CNPq for the financial support and is very grateful to H. Bursztyn for his generous tutelage. The author also thanks the referee whose comments and suggestions substantially improved this work.
\printbibliography
\end{document}